\documentclass{amsart}
\usepackage{amscd,amssymb,amsopn,amsmath,amsthm,graphics,amsfonts,accents,enumerate,verbatim,calc}
\usepackage[dvips]{graphicx}
\usepackage[colorlinks=true,linkcolor=red,citecolor=blue]{hyperref}
\usepackage[all]{xy}
\usepackage{tabu}
\usepackage{tikz}
\usepackage{mathdots}

\calclayout

\newcommand{\rt}{\rightarrow}
\newcommand{\lrt}{\longrightarrow}

\newcommand{\st}{\stackrel}

\newcommand{\la}{\lambda}
\newcommand{\La}{\Lambda}
\newcommand{\Ga}{\Gamma}

\newcommand{\CA}{\mathcal{A} }
\newcommand{\CC}{\mathcal{C} }
\newcommand{\DD}{\mathcal{D} }
\newcommand{\CE}{\mathcal{E}}

\newcommand{\CI}{\mathcal{I} }

\newcommand{\CP}{\mathcal{P} }

\newcommand{\CX}{\mathcal{X} }
\newcommand{\CY}{\mathcal{Y} }
\newcommand{\CW}{\mathcal{W}}
\newcommand{\CV}{\mathcal{V}}
\newcommand{\CU}{\mathcal{U}}

\newcommand{\CH}{\mathcal{H}}

\newcommand{\mmod}{{\rm{{mod\mbox{-}}}}}

\newcommand{\Hom}{{\rm{Hom}}}
\newcommand{\Ext}{{\rm{Ext}}}

\newcommand{\bsm}{\begin{smallmatrix}}
	\newcommand{\esm}{\end{smallmatrix}}
\newcommand{\bbm}{\begin{matrix}}
	\newcommand{\ebm}{\end{matrix}}

\theoremstyle{plain}
\newtheorem{theorem}{Theorem}[section]
\newtheorem{corollary}[theorem]{Corollary}
\newtheorem{lemma}[theorem]{Lemma}

\newtheorem{proposition}[theorem]{Proposition}

\theoremstyle{definition}
\newtheorem{definition}[theorem]{Definition}
\newtheorem{example}[theorem]{Example}

\newtheorem{construction}[theorem]{Construction}

\newtheorem{remark}[theorem]{Remark}

\theoremstyle{plain}

\theoremstyle{definition}

\numberwithin{equation}{section}

\begin{document}

\title[ the  morphism category of projective modules]{The Auslander-Reiten theory of the  morphism category of projective modules}

\author[     Rasool Hafezi and Jiaqun Wei   ]{    Rasool Hafezi and Jiaqun Wei}
\dedicatory{}
\address{School of Mathematics and Statistics,
Nanjing University of Information Science \& Technology, Nanjing 210044, P.\,R. China}
\email{hafezi@nuist.edu.cn}
\address{School of Mathematics Science,
Nanjing Normal University, Nanjing 210023, P.\,R. China}
\email{weijiaqun@njnu.edu.cn}



\begin{abstract}
We investigate the structure of certain almost split sequences in $\CP(\La)$, i.e., the category of morphisms between projective modules over an Artin algebra $\Lambda$. The category $\CP(\La)$ has very nice properties and is closely related to $\tau$-tilting theory, $g$-vectors, and Auslander-Reiten theory. We provide explicit constructions of certain almost split sequences ending at or starting from  certain objects. Applications, such as to $g$-vectors, are given. As a byproduct, we also show that there exists an injection from Morita equivalence classes of Artin algebras to equivalence classes of 0-Auslander exact categories.
\end{abstract}

\maketitle
\section{Introduction}\label{Introduction}
Almost split sequences, or Auslander-Reiten sequences, have played an important role in the study of representation theory of Artin algebras since their introduction. Almost split sequences have provided crucial insights into the representation theory of finite dimensional algebras by revealing connections between indecomposable modules. The existence of almost split sequences, or Auslander-Reiten triangles, has been studied in various contexts: abelian categories \cite[Chapter V]{ARS}, exact categories \cite{DRSSK, GR}, triangulated categories \cite{Ha}, extriangulated categories \cite{INP} and tri-exact categories \cite{LN}.

In \cite{HE}, one author of the current paper together with H. Eshraghi studied the structure of almost split sequences with certain ending terms in $\CH(\La)$, the morphism category of an Artin algebra $\La$. Some middle terms of certain almost split sequences are also discussed in the paper.

In this paper, we aim to understand the detailed construction of certain almost split sequences in $\CP(\La)$, i.e., the category of morphisms between projective modules over an Artin algebra $\Lambda$. The category $\CP(\La)$ is an extension closed functorially finite subcategory of the morphism category $\CH(\La)$ and has almost split sequences.

 Studying $\CP(\La)$ instead of $\mmod \La$, where we deal only with projective modules, has some benefits. First, there is a natural duality between $\CP(\La)$ and $\CP(\La^{\rm op})$ induced by the functor $(-)^*=\Hom_{\La}(-,\La)$. This may provide new insights. Second, the Nakayama functor induces an equivalence  between   $\CP(\La)$ and the  morphism category  $\CI(\La)$ of injective modules, see Section \ref{Section4} for more details.  The category  $\CP(\La)$ is closely related to almost split sequences since the minimal projective presentation of a module is an object in $\CP(\La)$ and this minimal projective presentation is the basis to define the transpose of the module,  necessary in the orginal study of almost split sequences \cite{ARS,ASS}. The category  $\CP(\La)$ is also related to  $g$-vectors, as the definition of the $g$-vector is given by terms of the minimal projective presentations. Moreover, the $\tau$-tilting theory and two-term silting complexes are naturally related to $\CP(\La)$.

The paper is organised as follows.

In Section 3, we study $\CP(\La)$-approximations of objects in $\CH(\La)$. For certain objects in $\CH(\La)$, we provide   explicit formulas for the (minimal) right and  left $\CP(\La)$-approximations.

In Section 4, we first define the concept of 0-Auslander exact categories, which is inspired by the notion of 0-Auslander extriangulated categories recently introduced in \cite{GNP2}. We will demonstrate that there exists an injection from Morita equivalence classes of Artin algebras to equivalence classes of 0-Auslander exact categories. We will also establish that the cokernel functor behaves well respect to almost split sequences. As a result, we establish a close connection between the AR-quiver of $\CP(\La)$ and that of  $\La$. In the reminder of this section, we will study the morphism category of injective modules. This enables us to view some known equivalences in terms of morphism categories.

In Section 5, we determined explicitly the almost split sequences in $\CP(\La)$ ending at or starting from objects of the
form  ($0\to P$) or ($P\to 0$). These results will be helpful in constructing the Auslander-Reiten quiver of $\CP(\La)$, particularly when the knitting procedure works
well. As an application, we will provide additional information for the canonical map introduced in Remark \ref{canonicalRemark}

In the last section, we collect some consequence for mod-$\La$ by making use of our study of the morphism category $\CP(\La)$. In particular, we will provide some results about g-vectors, which are important in $\tau$-tilting theory as introduced
in \cite{AIR}.

\section{Preliminaries}\label{Preleminary}
In this section, for the convenience of the reader, we present some definitions and results that will be used throughout the paper.

\subsection{Factor category}
Let $\DD$ be a class of objects of a category $\CC$.
Then  ideal of $\CC$ formed by all maps factoring through
a finite direct sum of objects in $\DD$ is denoted by $<\DD>$. The factor category $\frac{\CC}{<\DD>}$ has the same objects as  $ \CC$. For any  $X, Y \in \CC$, we define the hom-set
\[
\frac{\CC}{<\DD>}(X, Y ) := \frac{\mathcal{C}(X, Y )} {<\DD>(X, Y )}.
\]
For ease of notation, we follow  \cite{RZ} and  write $\frac{\CC}{\DD}$  or  $\CC/\DD$  instead of $\frac{\CC}{<\DD>}$.

\subsection{Morphism category}

Let $\CC$ be a category.
The morphism category $\rm{H}(\CC)$ of $\CC$ is a category whose objects are morphisms $f:X\rt Y$ in $\CC$, and whose morphisms are given by commutative diagrams. If we regard the morphism $f:X\rt Y$ as an object in $\rm{H}(\CC)$, we will usually present it as
 $(X\st{f}\rt Y)$ or
 $\left(\begin{smallmatrix}
 X\\
 Y
 \end{smallmatrix}
 \right)_f$.
A morphism between the objects $(X\st{f}\rt Y)$ and $(X'\st{f'}\rt Y')$ is presented as
$(\sigma_1, \sigma_2): (X\st{f}\rt Y)\rt (X'\st{f'}\rt Y')$
or,   $\left(\begin{smallmatrix}
 \sigma_1\\
 \sigma_2
 \end{smallmatrix}
 \right): \left(\begin{smallmatrix}
 X\\
 Y
 \end{smallmatrix}
 \right)_f \rt \left(\begin{smallmatrix}
 X'\\
 Y'
 \end{smallmatrix}
 \right)_{f'}$, where $\sigma_1:X\rt X'$ and $\sigma_2:Y\rt Y'$ are morphisms in $\CC$ with $\sigma_2f=f'\sigma_1.$

\vspace{.1 cm}
For breviary,  we let $\CH(\La):={\rm H}(\mmod \La)$ and $\CP(\La):={\rm H}({\rm prj}\mbox{-}\La)$. If there is no possibility of  confusion, we simply write $\CH$ and $\CP$, respectively. It is known that $\mmod T_2(\La)\simeq \CH(\La)$, where $ T_2(\La)$ is upper triangular matrix algebra over $\La$. We can freely apply this equivalence to use results from the module category in  the abelian category $\CH.$
\subsection{Relative Auslander-Reiten theory}\label{relative AR-theory}
Following \cite{K},  we have the notion of splitting projective modules and Ext-projective modules in a subcategory $\CX$ of $\mmod \La.$ Specifically, a module $X$ in $\CX$ is called Ext-projective if $\Ext^1_{\La}(X, Y)=0$ for every $Y$ in $\CX.$  A module $Y $ in $\CX$ is said to be splitting-projective if every epimorphism $X\rt Y$ in $\mmod \La$ with $X \in \CX$ splits. The notions of splitting injective and Ext-injective are defined dually.

For all $Z \in \mmod \La,$ we denote by $f^Z:Z\rt X^Z$ a minimal left $\CX$-approximation (if it exists), and by $g_Z:X_Z\rt Z,$ a minimal right $\CX$-approximation of $Z$.
\begin{theorem}\cite[Corollary 5.4]{KP}\label{KP-ext-proj}
 Let $\CX$ be contravariantly finite and extension-closed subcategory in $\mmod \La$. Let $C \in \CX$ be indecomposable and not Ext-projective,  and let $\lambda: 0 \rt \tau_{\La} C \rt E \rt C \rt 0$ be an almost split sequence in $\mmod \La$.
 \begin{itemize}
     \item [$(a)$] There exists a unique,  up to isomorphism,  exact commutative diagram
     	$$\xymatrix{		
		\alpha: \ \ 0\ar[r]&	X_{\tau_{\La}C}\ar[d]^{g_{\tau_{\La}C}} \ar[r] & X_E
			\ar[d]^{g_E} \ar[r] & C \ar@{=}[d]\ar[r] &0 \\ \lambda: \ \		0 \ar[r]&	\tau_{\La} C\ar[r] & E
			\ar[r] &  C \ar[r]&0 }	$$
     \item [$(b)$]The exact sequence $\alpha$ is isomorphic to  a direct sum of an  almost split sequence $0 \rt \tau_{\CX} C \rt X \rt C \rt 0$ in $\CX$ and the split exact sequence $0 \rt Y \st{1} \rt Y \rt 0 \rt 0 $ with Ext-injective module $Y$ in $\CX.$

     \item [$(c)$] $\alpha$ is an almost split sequence in $\CX$ if and only if $X_{\tau_{\La} C}$ is indecomposable.

 \end{itemize}
\end{theorem}

We refer to \cite[Corollary 5.5]{KP} for  a dual version of the above theorem.

\subsection{certain almost split sequences in $\CH$}

We will present  alternative proofs for some results from \cite{HE}, which are more constructive.

\begin{proposition}\label{Reminder:0M}
	Let $(P\st{f}\rt Q)$ be an indecomposable non-projective object in $\CP$.  We set $M={\rm Cok}(P\st{f}\rt Q)$, $p_0:Q\rt M$ as the  canonical quotient map, and the epi-mono factorization  $f:P\st{p_1}\rt \Omega_{\La}(M)\st{i}\rt Q$ of the morphism $f$. Consider the following pull-back diagram:
	
	$$\xymatrix{		  & \Omega_{\La}(M) \ar[d]^h
			\ar@{=}[r] & \Omega_{\La}(M) \ar[d]^i  \\
			\tau_{\La} M\ar@{=}[d] \ar[r]^b & \tau_{\La} M \oplus Q
			\ar[d]^d \ar[r]^a & Q \ar[d]^{p_0}  \\
			\tau_{\La}M   \ar[r]^{k} & B
			\ar[r]^{g} & M }	$$
	where the lowest row is an almost split sequence in $\mmod \La.$ Then we have the following short exact sequence in $\CH$
	
	$$\xymatrix@1{  0\ar[r] & {\left(\begin{smallmatrix} 0\\ \tau_{\La}M\end{smallmatrix}\right)}_{0}
			\ar[rr]^-{\left(\begin{smallmatrix} 0 \\ b \end{smallmatrix}\right)}
			& & {\left(\begin{smallmatrix} P \\ \tau_{\La}M\oplus Q \end{smallmatrix}\right)}_{hp_1}\ar[rr]^-{\left(\begin{smallmatrix} 1 \\ a\end{smallmatrix}\right)}& &
			{\left(\begin{smallmatrix}P \\ Q\end{smallmatrix}\right)}_{f}\ar[r]& 0. } \ \    $$
\noindent	
	 In particular,  $\tau_{\mathcal{H}}(P\st{f}\rt Q)= (0 \rt \tau_{\La} M)$.
\end{proposition}

\begin{proof}
By  \cite[Proposition 2.4]{HE}, we have $\tau_{\CH}(P\st{f}\rt Q)=(0\rt \tau_{\La}M)$.  Using \cite[\S V, Proposition 2.2]{AuslanreitenSmalo}, it suffices to show that any non-isomorphism
$\left(\begin{smallmatrix}
0 \\\phi
\end{smallmatrix}\right):\left(\begin{smallmatrix}
0\\ \tau_{\La}M
\end{smallmatrix}\right)_{0}\rt \left(\begin{smallmatrix}
0\\ \tau_{\La}M
\end{smallmatrix} \right)_{0}$ factors over the monomorphism  lying in the short exact sequence.

Since $\left(\begin{smallmatrix}
0 \\\phi
\end{smallmatrix}\right)$ is a non-isomorphism in $\CH$, it follows that   $\phi$ is also  a non-isomorphism in $\mmod \La$. Thus,   there exists a morphism $s:B\rt \tau_{\La} M$ such that $sk=\phi$, as $\la$ is an almost split sequence in $\mmod \La$. Then, we can  verify directly that $\left(\begin{smallmatrix}
0 \\ sd
\end{smallmatrix}\right):\left(\begin{smallmatrix}
P\\ \tau_{\La} M \oplus Q
\end{smallmatrix}\right)_{hp_1}\rt \left(\begin{smallmatrix}
0\\\tau_{\La} M
\end{smallmatrix} \right)_{0}$ indeed gives  the desired factorization.
\end{proof}

\begin{proposition}\label{Reminder:MM}
	Let $(I\st{g}\rt J)$ be an indecomposable non-injective object in $\mathcal{H}$ with $I$ and $J$ injective. Set $N={\rm Ker}(I\st{g}\rt J)$, $i_0:N\rt I$ the canonical inclusion,  and the epi-mono factorization  $g:I\st{p_0}\rt \Omega^{-1}_{\La}(M)\st{i_1}\rt J$ of the morphism $g$. Consider the following push-out diagram
	
	$$\xymatrix{	N\ar[r]\ar[d]^{i_0}	  & C \ar[d]\ar[r]
			 & \tau^{-1}_{\La}N \ar@{=}[d]  \\
			I\ar[d]^{p_0} \ar[r]^d & I\oplus \tau^{-1}_{\La} N
			\ar[d]^s \ar[r]^c & \tau^{-1}_{\La} N   \\
			  \Omega^{-1}_{\La}N  \ar@{=}[r] & \Omega^{-1}_{\La} N
			 & }	$$
	where the top row is an almost split sequence in $\mmod \La.$ Then we have the following short exact sequence in $\CH$
	
	$$\xymatrix@1{  0\ar[r] & {\left(\begin{smallmatrix} I\\ J\end{smallmatrix}\right)}_{g}
			\ar[rr]^-{\left(\begin{smallmatrix} d \\ 1 \end{smallmatrix}\right)}
			& & {\left(\begin{smallmatrix} I\oplus \tau^{-1}_{\La}N \\ J \end{smallmatrix}\right)}_{i_1s}\ar[rr]^-{\left(\begin{smallmatrix} c \\ 0\end{smallmatrix}\right)}& &
			{\left(\begin{smallmatrix}\tau^{-1}_{\La}N\\ 0\end{smallmatrix}\right)}_{0}\ar[r]& 0. } \ \    $$
	\noindent
	 In particular,  $\tau^{-1}_{\mathcal{H}}(I\st{g}\rt J)= ( \tau^{-1}_{\La} N \rt 0)$.
\end{proposition}
\begin{proof}
According to \cite[Proposition 2.2]{HE}, we have $\tau_{\CH}(0 \rt \tau^{-1} N)= (\nu P_1\st{\nu \alpha}\rt  \nu P_0)$, where $P_1\st{\alpha}\rt  P_0 \rt \tau^{-1} N\rt 0$ is a minimal projective presentation.  By applying  the following exact sequence (see \cite[\S III, Proposition 5.3]{SY}):
$$0 \rt N \rt \nu P_1 \st{\nu \alpha} \rt  \nu P_0 \rt \nu \tau^{-1}_{\La} N\rt 0$$
\noindent
and using this fact that $\nu:{\rm prj}\mbox{-}\La \rt {\rm inj}\mbox{-}\La $ is an  equivalence, we can deduce that $0 \rt N\rt \nu P_1\st{\nu \alpha}\rt  \nu P_0$ is a minimal injective presentation. Hence, by the uniqueness of the minimal injective presentation, we may identify $(I\st{g} \rt  J)=(\nu P_1\st{\nu \alpha}\rt \nu P_1)$. Therefore, $\tau^{-1}_{\CH}(I\st{g}\rt I)=(\tau^{-1}_{\La}N \rt 0)$. Then,  due to  \cite[\S V, Proposition 2.2]{AuslanreitenSmalo}, it suffices to show that any non-isomorphism
$\left(\begin{smallmatrix}
\psi \\ 0
\end{smallmatrix}\right):\left(\begin{smallmatrix}
 \tau^{-1}_{\La} N \\ 0
\end{smallmatrix}\right)_{0}\rt \left(\begin{smallmatrix}
 \tau^{-1}_{\La} N\\ 0
\end{smallmatrix} \right)_{0}$ factors over the epimorphism $\left(\begin{smallmatrix}
 c \\ 0
\end{smallmatrix} \right).$ This can be verified similarly to  Proposition \ref{Reminder:0M}.
\end{proof}
As the above propositions state,  one can  obtain  certain  almost split sequence in $\CH$ by performing  only   pull-back and push-out operations.
\begin{proposition}\label{proj-nakaya}
Let $Q$ be an indecomposable projective module. There is the following almost split sequence in $\CH$
	$$\xymatrix@1{  0\ar[r] & {\left(\begin{smallmatrix} 0\\ \nu Q\end{smallmatrix}\right)}_{0}
			\ar[rr]^-{\left(\begin{smallmatrix} 0\\ 1\end{smallmatrix}\right)}
			& & {\left(\begin{smallmatrix}Q\\ \nu Q \end{smallmatrix}\right)}_{f}\ar[rr]^-{\left(\begin{smallmatrix} 1 \\ 0\end{smallmatrix}\right)}& &
			{\left(\begin{smallmatrix}Q\\0\end{smallmatrix}\right)}_{0}\ar[r]& 0. }$$
In particular, $\tau_{\CH}(Q\rt 0)=(0 \rt \nu Q).$			
			
\end{proposition}
\begin{proof}
We know from \cite[Proposition 2.2]{HE} that $\tau_{\CH}(Q\rt 0)= (0 \rt \nu Q)$, which forces us  to have an almost split sequence as in the statement.
\end{proof}
\begin{remark}\label{canonicalRemark}
An interesting aspect  of  Proposition \ref{proj-nakaya} is that  indecomposable modules $Q$ and $\nu Q$ are connected  by  a morphism that arises from  the middle term of an almost split in $\CH$. The morphism $f:Q\rt \nu Q$ that appears in the almost split sequence is denoted by  $\alpha_Q$ in the subsequent discussion. This notation is meaningful because $\alpha_Q$ is  unique up to isomorphism, due to  the uniqueness of the almost split sequence.
The morphism $\alpha_Q$ can be given canonically as follows: Since $\nu Q$ is an indecompsable injective module and $\alpha_Q$ a non-zero morphism, we can deduce that ${\rm top}(Q)\simeq {\rm soc}(\nu Q)$. The non-retraction $(1, ~0):(Q\st{p}\rt {\rm top}(Q) )\rt (Q\rt 0)$ factors through the epimorphism lying in the almost split sequence given  in Proposition \ref{proj-nakaya}. Then, by taking into account the isomorphism  ${\rm top}(Q)\simeq {\rm soc}(\nu Q)$,  we obtain the following factorization

 \[ \xymatrix@R-1.5pc { && Q\ar[rd]_>>>>>>>>>>{p}\ar[rr]^{\alpha_Q}&& \nu Q& &  \\ && & {\rm top}(Q)\simeq {\rm soc}(\nu Q)\ar[ru]_>>>>>>>{i}&&& }\]
Here,  $p$ and $i$ are the canonical epimorphism and inclusion, respectively.
\end{remark}






\section{ left (right)  $\CP(\La)$-approximations of certain objects}
In this section, we  provide  an explicit formula for a minimal right, resp. left, $\CP$-approximation of some certain objects in $\CH.$

\hspace{ 1 mm}

 We have an exact structure on $\CP$ inherited from $\CH$, since $\CP$ is a   extension closed  subcategory. Throughout the paper, we consider  the additive category  $\CP$   as an exact category by the inherited exact structure. It is easy to check  that an object ${\rm P}$ in $\CP$ is relative projective  in the exact category $\CP$ if and only it is splitting projective in the subcategory $\CP$ and if and only if it is Ext-projective in the subcategory $\CP$, see \ref{relative AR-theory}. \\

\begin{lemma}\label{proj-inde-p}
Let $\rm P$ be an indecomposable object in the exact category $\CP$.
\begin{itemize}
\item [$(1)$]$\rm P$ is projective if and only if it is isomorphic to either  $(P\st{1}\rt P)$ or $(0 \rt P)$ for some projective indecomposable $\La$-module.
\item [$(2)$] $\rm P$ is injective if and only if it is isomorphic to either  $(P\st{1}\rt P)$ or $(P \rt 0)$ for some projective indecomposable $\La$-module.
\end{itemize}
\end{lemma}
\begin{proof}
The lemma is proved easily by using  the following two exact sequences,  which exist for every object $(P\st{d}\rt Q)$ in $\CP$
    $$\xymatrix@1{  0\ar[r] & {\left(\begin{smallmatrix} 0\\ P\end{smallmatrix}\right)}_{0}
			\ar[rr]^-{\left(\begin{smallmatrix} 0\\ \left[\begin{smallmatrix} d\\ 1\end{smallmatrix}\right]\end{smallmatrix}\right)}
			& &{\left(\begin{smallmatrix}0\\  Q\end{smallmatrix}\right)}_{0}\oplus {\left(\begin{smallmatrix}P\\  P\end{smallmatrix}\right)}_{1}\ar[rr]^-{\left(\begin{smallmatrix} 1 \\ [-1~~d]\end{smallmatrix}\right)}& &
			{\left(\begin{smallmatrix}P\\Q\end{smallmatrix}\right)}_{d}\ar[r]& 0 }$$
and
    $$\xymatrix@1{  0\ar[r] & {\left(\begin{smallmatrix} P\\ Q\end{smallmatrix}\right)}_{d}
			\ar[rr]^-{\left(\begin{smallmatrix}  \left[\begin{smallmatrix} d\\ 1\end{smallmatrix} \right] \\ 1 \end{smallmatrix}\right)}
			& &{\left(\begin{smallmatrix}Q\\  Q\end{smallmatrix}\right)}_{1}\oplus {\left(\begin{smallmatrix}P\\  0\end{smallmatrix}\right)}_{0}\ar[rr]^-{\left(\begin{smallmatrix}  [-1~~d] \\ 0\end{smallmatrix} \right)}& &
			{\left(\begin{smallmatrix}Q\\0\end{smallmatrix}\right)}_{0}\ar[r]& 0. }$$
\end{proof}
The above lemma also follows from \cite[Corollary 3.3]{BSZ} by taking $n=2.$
\subsection{right  minimal $\CP$-approximation}

\begin{lemma}\label{righ-app-certan-obj}
Let $M$ be an indecomposable module. Let  $P_1\st{f}\rt P_0\st{\sigma} \rt  M \rt 0$ be  a minimal projective presentation of $M$. Then the following statements hold.
\begin{itemize}
    \item [$(1)$] The morphism $\left(\begin{smallmatrix} 0\\ \sigma \end{smallmatrix}\right):\left(\begin{smallmatrix} P_1\\ P_0\end{smallmatrix}\right)_f\rt \left(\begin{smallmatrix} 0\\ M\end{smallmatrix}\right)_0 $ is a right minimal $\CP$-approximation.
    \item [$(2)$] The morphism $\left(\begin{smallmatrix} \sigma\\ \sigma \end{smallmatrix}\right):\left(\begin{smallmatrix} P_0\\ P_0\end{smallmatrix}\right)_1\rt \left(\begin{smallmatrix} M\\ M\end{smallmatrix}\right)_1 $ is a right minimal $\CP$-approximation.
    \item [$(3)$] The morphism $\left(\begin{smallmatrix} \sigma\\ 0 \end{smallmatrix}\right):\left(\begin{smallmatrix} P_0\\ 0 \end{smallmatrix}\right)\rt \left(\begin{smallmatrix} M\\ 0 \end{smallmatrix}\right)$ is a right minimal $\CP$-approximation.
\end{itemize}

\end{lemma}
\begin{proof}
We only prove  $(1)$ as  proofs of the  other statements are rather similar. The proof of the minimality  of the morphism $\left(\begin{smallmatrix} 0\\ \sigma \end{smallmatrix}\right)$ is a direct consequence of the fact that $P_1\st{f}\rt P_0\st{\sigma}\rt M\rt 0$ is a minimal projective presentation of $M$. We will  show that  $\left(\begin{smallmatrix} 0\\ \sigma \end{smallmatrix}\right)$ is a right $\CP$-approximation of $M$. Assume a morphoism $\left(\begin{smallmatrix} 0\\ \psi \end{smallmatrix}\right):\left(\begin{smallmatrix} P\\ Q\end{smallmatrix}\right)_g\rt \left(\begin{smallmatrix} 0\\ M\end{smallmatrix}\right)_0 $ is given. Since $\sigma$ is a projective cover of $M$, there exists a morphism $d:Q\rt M$ such that $\sigma d=\psi.$ On the other hand, as $\sigma dg=\psi g=0$, there exists a morphism $\overline{d}:P\rt \Omega_{\La}(M)$ which satisfies the following commutative diagram
  \[ \xymatrix@R-1.5pc {  &    &P\ar[r]^{g}\ar@/_1pc/@{.>}[dddl]_{\overline{d}} & Q\ar[dd]^{\psi}\ar[ldd]_d \\  &&  & & \\ P_1\ar[rd]^{p}\ar[rr]^{f}&& P_0\ar[r]^{\sigma}& M\ar[r] & 0  \\&\Omega_{\La}(M)\ar[ru]^{i}&&& }\]
Because of the projective cover $p:P_1\rt \Omega_{\La}(M)$, we get the morphism $s:P\rt P_1$ with $ps_1=\overline{d}$. Now it is easily seen that the morphism  $\left(\begin{smallmatrix} s\\ d \end{smallmatrix}\right):\left(\begin{smallmatrix} P\\ Q\end{smallmatrix}\right)_g\rt \left(\begin{smallmatrix} P_1\\ P_0\end{smallmatrix}\right)_f $ provides the desired factorization.
\end{proof}
The following result provides a structural method for constructing  a right $\CP$-approximation of an given object in $\CH.$
\begin{proposition}\label{approx-cp}
Let $(N\st{g}\rt M)$ be  an object in $\CH$.  Assume $P_1\st{f}\rt P_0 \st{\sigma}\rt M\rt 0$ is a  projective presentation of $M$, and $Q\st{e}\rt N\rt 0$ is an epimorphism with a projective module $Q$. Further,  consider the following commutative diagram:
  \[\xymatrix{0\ar[r]&P_1 \ar[r]^{[1~~0]^t} \ar[d]^{p} & P_1\oplus Q\ar[r]^{[0~~1]} \ar[d]^{[p~~e']} 				 & Q \ar[r] \ar[d]^e\ar@{.>}[ld]^<<<{e'} & 0 \\0 \ar[r]&
        		\Omega_{\La}(M)\ar[r]^{j}\ar@{=}[d]
         & U \ar[r]^{h}\ar[d]^d
        			 & N \ar[r]\ar[d]^g  & 0  \\0\ar[r] &\Omega_{\La}(M)\ar[r]^i&P_0\ar[r]^{\sigma}& M\ar[r]& 0
        }\]
\noindent
where the rightmost square  is a pullback of $g $ along $\sigma$,   $f=ip$ an epi-mono factorization of $f$, and $e':Q\rt U$ a lifting of the morphism $e$, meaning that  $he'=e.$ Then

 $$\left(\begin{smallmatrix} [0~~e]\\ \sigma \end{smallmatrix}\right):\left(\begin{smallmatrix} P_1\oplus Q\\ P_0\end{smallmatrix}\right)_{[f~~de']}\rt \left(\begin{smallmatrix} N\\ M\end{smallmatrix}\right)_g $$
is a right $\CP$-approximation.
\end{proposition}
\begin{proof}
Consider  the following short exact sequence in $\CH$
	$$\xymatrix@1{ \eta: \ \ 0\ar[r] & {\left(\begin{smallmatrix} V\\  \Omega_{\La}(M)\end{smallmatrix}\right)_q}
			\ar[rr]
			& & {\left(\begin{smallmatrix}P_1\oplus Q\\ P_0 \end{smallmatrix}\right)}_{[f~~de']}\ar[rr]^-{\left(\begin{smallmatrix} [0~~e] \\ \sigma\end{smallmatrix}\right)}& &
			{\left(\begin{smallmatrix}N\\M\end{smallmatrix}\right)}_{g}\ar[r]& 0 }.$$
As the direct summand $P_1$ is mapped  to zero by $[0~~e]$, it follows that $P_1 \subseteq V$. This implies  that  the morphism $q$ is surjective.
In view of the following short exact sequence,   we can deduce that $(V\st{q}\rt \Omega_{\La}(M))$ lies in $\CP^{\perp}$:
$$\xymatrix@1{  \ \ 0\ar[r] & {\left(\begin{smallmatrix} {\rm Ker}(q)\\ 0 \end{smallmatrix}\right)_0}
			\ar[rr]
			& & {\left(\begin{smallmatrix} V\\ \Omega_{\La}(M) \end{smallmatrix}\right)}_{q}\ar[rr]^-{\left(\begin{smallmatrix} q \\ 1\end{smallmatrix}\right)}& &
			{\left(\begin{smallmatrix}\Omega_{\La}(M)\\ \Omega_{\La}(M)\end{smallmatrix}\right)}_{1}\ar[r]& 0 }.$$
This means that the morphism $([0~~e], \sigma)$ given in the statement plays the  role of  a right $\CP$-approximation for the object $(N\st{g}\rt M)$.		
\end{proof}

\begin{proposition}\label{Prop.dom-proj.approx}
Let $g:Q\rt M$ be a projective cover and let $P_1\st{f}\rt P_0\st{\sigma}\rt M\rt 0$ be a minimal projective presentation of $M$. Then
 $$\left(\begin{smallmatrix} [0~~1]\\ \sigma \end{smallmatrix}\right):\left(\begin{smallmatrix} P_1\oplus Q\\ P_0\end{smallmatrix}\right)_{[f~~g']}\rt \left(\begin{smallmatrix} Q\\ M\end{smallmatrix}\right)_g, $$
where $g':Q\rt P_0$ is a lifting of $g$, i.e, $\sigma g'=g$, is a minimal right $\CP$-approximation.
\end{proposition}
\begin{proof}
From Proposition \ref{approx-cp},  we can see that the  morphism given in the statement  is a right $\CP$-approximation. Thus,  it suffices to show that it is minimal. Suppose there exists an endomorphism $(\alpha, \beta):(P_1\oplus Q\st{[f~~g']}\rt P_0)\rt (P_1\oplus Q\st{[f~~g']}\rt P_0)$ such that $([0~~1], \sigma)\circ (\alpha, \beta)=([0~~1],  \sigma)$. Our goal is  to prove  that this  endomorphism is an automorphism. First, we observe that from   the commutative diagram  induced by the endomorpshim in $\mmod \La$ that the endomorphism $\alpha:P_1\oplus Q\rt P_1\oplus Q $ has no image in $Q$ when we  restrict $\alpha$ to  $P_1.$ Therefore,  the matrix presentation of $\alpha$ is given as
$\left[\begin{smallmatrix} \alpha_1 & \alpha_3 \\ 0 & \alpha_2 \end{smallmatrix}\right]$,
where $\alpha_1, \alpha_2$ are  endomorphisms on $P_1$ and $Q$,  respectively, and $\alpha_3$ is a morphism from $Q$ to $P_1.$  The following commutative diagram in $\mmod \La$ helps us easy to visualize the maps we are dealing with:
\begin{equation*}
    \begin{gathered}
      \xymatrix@!=0.5pc{ {} & P_1 \, \ar@{=}'[d][dd]\ar[dl]_{\alpha_1}
          \ar@{->}[rr]^-{[1~~0]^t} & & P_1\oplus Q \ar'[d][dd]_-{[f~~g']}
          \ar@{->}[rr]^-{[0~~1]}\ar[dl]^{\alpha}
          & & Q \ar[dd]^g\ar@{=}[dl]
          \\
          P_1 \,
                    \ar@{=}[dd]
          \ar@{->}[rr]^(0.70){[1~~0]^t} & &
          P_1\oplus Q
          \ar[dd]_(0.70){[f~~g']}
          \ar@{->}[rr]^(0.70){[0~~1]} & & Q
          \ar[dd]^<<<<{g}
          \\
          {} & P_1 \, \ar@{->}'[r]^-{f}[rr]\ar[dl]_(0.40){\alpha_1} & &
          P_0 \ar@{->}'[r]^-{\sigma}[rr]\ar[dl]^(0.50){\beta}
          & & M
          \\
          P_1
          \ar[rr]_-{f} & & P_0
                    \ar[rr]_-{\sigma} & & M\ar@{=}[ur]
                  }
    \end{gathered}
    \end{equation*}
By the minimality of projective presentation,  we can deduce that $\alpha_1$ and $\beta$ are authomorphisms. Moreover, we can  see that $g\alpha_2=g$, which implies that  $\alpha_2$ is also an authomorphism. Finally, the diagonal entries of $h$ are authomorphism, we can conclude that   $h$ is also  an authomorphism. So, we have completed our proof.
\end{proof}
\begin{proposition}\label{Prop-codm-prj}
Let $g: M \rt Q$ be a minimal morphism with a projective module $Q,$ and let $f: P\rt M\rt 0$ be a projective cover. Then
 $$\left(\begin{smallmatrix} f\\ 1 \end{smallmatrix}\right):\left(\begin{smallmatrix} P\\ Q\end{smallmatrix}\right)_{gf}\rt \left(\begin{smallmatrix} M\\ Q\end{smallmatrix}\right)_g $$
 is a minimal right $\CP(\La)$-approximation.
\end{proposition}
\begin{proof}

From Proposition \ref{approx-cp},  we know that the morphism $(f, 1)$ is a right $\CP$-approximation. Thus, it remains  to show that it is minimal. Assume  $(g_1, g_2 )$ is an endomorphism of $(P\st{gf}\rt  Q)$ satisfying $(g_1, g_2)\circ (f, 1)=(g_1, g_2)$. By examining  the  commutative diagrams that follow from the equality, we can deduce that  $g_2=1$ and $fg_1=g_1$. Since $f$ is minimal,  $g_1$ must be  an isomorphism. Therefore, we have established the desired minimality.
\end{proof}

\begin{proposition}
The subcategory $\CP$ is functorially finite in $\CH$.
\end{proposition}
\begin{proof}
As $\CP$ is resolving, thanks to \cite[Corollary 03]{KS}, it is enough to prove  that $\CP$ is contravariantly finite. This follows from  Proposition \ref{approx-cp}.
\end{proof}
 The above result, combined  with \cite[Theorem 2.4]{AS},    implies that  $\CP$ has almost split sequences. This can also be proven by using   \cite[Theorem 4.5]{BSZ}, where we take  the dualizing variety $\CA={\rm prj}\mbox{-}\La$ and set $n=2$, or by \cite[Theorem 5.2]{Ba}.

 \subsection{left minimal $\CP$-approximation} There is a dual version of our  results concerning left minimal $\CP$-approximation. To introduce the dual,  we first  need to define the notion of a minimal projective copresentation.

\begin{definition}
Let $M$ be a module. An exact sequence $M\st{\delta}\rt P \st{g}\rt  Q$ is said to be a minimal projective copresentation of $M$ if it satisfies the following conditions:
\begin{itemize}
    \item [$(1)$] The morphism $\delta:M\rt P$ is a minimal left ${\rm prj}\mbox{-}\La$-approximation of $M$;
    \item [$(2)$] Let $q:{\rm Cok} \ \delta\rt Q $ be  a minimal left ${\rm prj}\mbox{-}\La$-approximation of ${\rm Cok} \ \delta$. The morphism $g$ is followed by the morphism $q$ and the canonical map $P\rt {\rm Cok} \ \delta$.
\end{itemize}
 If the   minimality condition on  the  morphisms is not necessary, the the sequence is simply called a projective copersentation.
\end{definition}
\begin{lemma}\label{righ-app-certan-obj2}
Let $M$ be an indecomposable module. Let  $M\st{\delta}\rt Q_0\st{g} \rt Q_1 $ be  a minimal projective copresentation of $M$. Then the following statements hold.
\begin{itemize}
    \item [$(1)$] The morphism $\left(\begin{smallmatrix} \delta\\ 0 \end{smallmatrix}\right):\left(\begin{smallmatrix} M\\ 0\end{smallmatrix}\right)_0\rt \left(\begin{smallmatrix} Q_0\\ Q_1\end{smallmatrix}\right)_0 $ is a left minimal $\CP$-approximation.
    \item [$(2)$] The morphism $\left(\begin{smallmatrix} \delta\\ \delta \end{smallmatrix}\right):\left(\begin{smallmatrix} M\\ M\end{smallmatrix}\right)_1\rt \left(\begin{smallmatrix} Q_0\\ Q_0\end{smallmatrix}\right)_1 $ is a left minimal $\CP$-approximation.
    \item [$(3)$] The morphism $\left(\begin{smallmatrix} 0\\ \delta \end{smallmatrix}\right):\left(\begin{smallmatrix} 0\\ M \end{smallmatrix}\right)\rt \left(\begin{smallmatrix} 0\\ Q_0 \end{smallmatrix}\right)$ is a left minimal $\CP$-approximation.
\end{itemize}
\end{lemma}
\begin{proof}
The lemma is  proven using  dual arguments to the ones used in Lemma \ref{righ-app-certan-obj}.
\end{proof}
\begin{proposition}\label{approx-cp2}
Let $(N\st{g}\rt M)$ be  an object in $\CH$ and  $\delta: N\rt Q_0$   a left $\CP$-approximation of  $N$. Consider the following  commutative diagram,  which is obtained by taking  push-out of $\delta$ along $g$:
  \[\xymatrix{&
        		N\ar[r]^{\delta}\ar[d]^g
         & Q_0 \ar[d]^h
        			 &     \\ &M\ar[r]^{\la}&U.&
        }\]
\noindent
 Then the following morphism is a left $\CP$-approximation,

 $$\left(\begin{smallmatrix} \delta \\ d \la \end{smallmatrix}\right):\left(\begin{smallmatrix} N\\ M\end{smallmatrix}\right)_{g}\rt \left(\begin{smallmatrix} Q_0\\ Q\end{smallmatrix}\right)_{dh} $$
where $d:U \rt Q$ is a  left  ${\rm prj}\mbox{-}\La$-approximation.
\end{proposition}
\begin{proof}
Let  $\left(\begin{smallmatrix} \sigma_1 \\ \sigma_2 \end{smallmatrix}\right):\left(\begin{smallmatrix} N\\ M\end{smallmatrix}\right)_{g}\rt \left(\begin{smallmatrix} P_0\\ P_1\end{smallmatrix}\right)_{s}$  be a morphism in $\CH$ with $P_0$ and $P_1$ projective modules. Since $\delta$ is a left $\CP$-approximation, there exists a morphism $a:Q_0\rt P_0$ such that $a\delta=\sigma_1.$ Moreover, the push-down property implies that  the morphism $l:U\rt P_1$ satisfies the following commutative diagram:
 \[\xymatrix{&
        		N\ar[r]^{\delta}\ar[d]^g
         & Q_0 \ar[d]^h\ar[rdd]^{sa}
        			 &     \\ &M\ar[r]^{\la}\ar[rrd]_{\sigma_2}&U\ar@{.>}[dr]^<<<<<l&\\ &&&P_1&        }\]

 Again, using the fact  that $d$ is a left $\CP$-approximation, we obtain a morphism $r:Q\rt P_1$ such that $rd=l$. Then,   the morphism $$\left(\begin{smallmatrix} a \\ r \end{smallmatrix}\right):\left(\begin{smallmatrix} Q_0\\ Q\end{smallmatrix}\right)_{dh}\rt \left(\begin{smallmatrix} P_0\\ P_1\end{smallmatrix}\right)_{s}$$
    give us the desired factorization.
\end{proof}


In order to obtain  a precise dual of Proposition \ref{approx-cp},  we need to impose  additional assumption on $\La.$

\begin{proposition}\label{left-approximation-3.10}
Let  $\La$ be a selfinjective algebra and let $(N\st{f}\rt M)$ be  an object in $\CH$. Assume $0 \rt N \st{\sigma}\rt I_0\st{g}\rt I_1$ is a projective copresentation of $M$, and $0 \rt M \st{e}\rt I  $ is a monomorphism with a projective  module I. Further,  consider the following commutative diagram such that the rightmost square  is pushout of $f $ along $\sigma$, and
  \[\xymatrix{0\ar[r]&N \ar[r]^{\sigma} \ar[d]^{f} & I_0\ar[r]^{p} \ar[d]^r & \Omega^{-1}_{\La}(N) \ar[r] \ar@{=}[d]\ar & 0 \\0 \ar[r]&
        		M\ar[r]^{j}\ar[d]^e
         & U \ar[r]^{h}\ar[d]^d\ar@{.>}[dl]^{e'}
        			 & \Omega^{-1}(N) \ar[r]\ar[d]^i & 0  \\0\ar[r] &I\ar[r]_{[1~~0]^t}& I\oplus I_1\ar[r]_{[0~~1]}& I_1\ar[r]& 0
        }\]
\noindent
where $g=ip$ an epi-mono factorization of $g$, and $e':U\rt I$ an extension of the morphism $e$, that is, $e'j=e.$ Then

 $$\left(\begin{smallmatrix} \sigma\\ [e~~0]^t\end{smallmatrix}\right):\left(\begin{smallmatrix} N\\ M\end{smallmatrix}\right)_{f}\rt \left(\begin{smallmatrix} I_0\\ I \oplus I_1\end{smallmatrix}\right)_{dr} $$
is a left $\CP$-approximation.
\end{proposition}
The next result is a dual  of Proposition \ref{Prop-codm-prj}.
\begin{proposition}
    Let $(P\st{g}\rt M)$ be an object in $\CH$ with $P$ a projective module. Further, let $h:M\rt Q$ be a minimal left ${\rm prj}\mbox{-}\La$-approximation. Then
     $$\left(\begin{smallmatrix} 1\\ h \end{smallmatrix}\right):\left(\begin{smallmatrix} P\\ M\end{smallmatrix}\right)_{g}\rt \left(\begin{smallmatrix} P\\ Q\end{smallmatrix}\right)_{hg} $$
      is a minimal left $\CP$-approximation.
\end{proposition}
\begin{proof}
  Proposition \ref{left-approximation-3.10}  implies that  the morphism $(1, h)$ is a left $\CP$-approximation. Thus, it is enough to prove that  it is minimal. Let $(r_1, r_2):(P\st{hg}\rt Q)\rt (P\st{hg}\rt Q)$ be an endomorphism with $(r_1, r_2)\circ (1, h)=(r_1, r_2)$. The following commutative diagram records  all the relevant morphisms in $\mmod \La$:
\begin{equation*}
    \begin{gathered}
      \xymatrix@!=0.1pc{ {} & P \, \ar@{>}[dd]^-g\ar@{=}[dl]
          \ar@{=}[rr] & &  P\ar[dd]^->>>>>>{hg}\ar[llld]^>>>>>>>>>>>{r_1}
                    & &  \\
          P  \ar[dd]^->>>>>>>>{hg}  & &
                    & &          \\
          {} & M \ar[rr]^-{h}\ar[dl]_(0.40){h} & &    Q\ar[llld]^>>>>>>>>>>>{r_2}
          & &           \\       Q         & &  & &
                  }
    \end{gathered}
    \end{equation*}
  As the diagram shows, we have $r_1=1$ and $r_2h=h$. Since  $h$ is a  minimal morphism, it  follows that $r_2$ is an isomorphism. Consequently, $(r_1, r_2)$ is an isomorphism.
\end{proof}

\section{From  the   Ausalnder-Reiten theory of $\CP(\La)$ to   $\mmod \La$  }\label{Section4}

The notion of  being $0$-Auslander has recently been defined for   extriangulated categories in \cite{GNP2}. This inspired us to define the concept of $0$-Auslander exact categories. We will demonstrate  that there exists  an injection  from Morita equivalence classes of Artin algebras to equivalence classes of $0$-Auslander exact categories. We will establish  that the cokernel functor behaves well respect to almost split sequences. As a result, we establish a  close connection between the AR-quiver of $\CP$ and $\La.$ In the reminder of this section, we will study the morphism category of injective modules. This enables us to view some known equivalences  in terms of morphism categories.

 \begin{definition} \label{def:0-Auslander}
 An exact category $\CE$ is 0-Auslander if
 \begin{itemize}
  \item [$(i)$] it has enough projectives;
  \item [$(ii)$] it is hereditary:  for all $X \in \CE$ there exists a short exact sequence $0 \rt P_1\rt P_0\rt X\rt 0,$
  where $P_0,P_1$ are projective;
  \item  [$(iii)$]it has dominant dimension at least 1: for any projective object $P$, there is a short exact sequence $0 \rt P \rt Q \rt I\rt 0,$  where $I$ is injective and $Q$ is projective-injective.
 \end{itemize}
\end{definition}
We will show  below that the morphism  categories of  projective modules provide a large source for  $0$-Auslander exact categories.

The cokernel functor ${\rm Cok}:\CP(\La)\rt \mmod \La$
 is defined  by sending an object $(P\st{f}\rt Q)$ in $\CP(\La)$ to ${\rm Cok}(f )$. It induces an equivalence as stated in the following theorem. For a proof, we refer to \cite[Theorem 4.2]{EH} or \cite[Proposition 3.3]{Ba}.

\begin{theorem}\label{Thm-Cok-equiv}

The functor ${\rm Cok}:\CP(\La) \lrt \mmod \La$ is full, dense and objective.  Then, the cokernel functor  makes the following diagram commute.
\[\xymatrix{
\CP(\La)\ar[r]^{\rm Cok}\ar[d]^{\pi} & \mmod \La\\
\frac{\CP(\La)}{\CV}\ar[ur]_{\overline{\rm Cok}}   \\
}\]
\noindent
Here, $\pi$ is the natural quotient map and $\CV$ is the full subcategory of $\CP(\La)$ generated by all finite direct sums of objects of  the forms $(P\st{1}\rt P)$ or $(P\rt 0)$, where  $P$ is a projective module.
\end{theorem}

\begin{lemma}
    The exact category $\CP(\La)$ is $0$-Auslander.
\end{lemma}
\begin{proof}
Let ${\rm P}=(P\st{f}\rt Q)$ be an object in $\CP(\La)$.  We have  the following short exact sequence:
   $$\xymatrix@1{  0\ar[r] & {\left(\begin{smallmatrix} 0\\ P\end{smallmatrix}\right)}_{0}
			\ar[r]
			&  {\left(\begin{smallmatrix}0\\Q\end{smallmatrix}\right)}_{0}\oplus {\left(\begin{smallmatrix}P\\P\end{smallmatrix}\right)}_{1} \ar[r]&
			{\left(\begin{smallmatrix}P\\ Q\end{smallmatrix}\right)}_{f}\ar[r]& 0. } \ \    $$	
   By applying  the classification of projective objects in  $\CP(\La)$ given in Lemma \ref{proj-inde-p}, the above short exact sequence implies that ${\rm P}$ has projective dimension at most  one in the exact category $\CP$. Hence $\CP$ is hereditary. Moreover, the short exact sequence implies that $\CP$  has enough projectives,  satisfying the conditions $(i)$ and $(ii)$ of Definition \ref{def:0-Auslander}.  Lemma \ref{proj-inde-p} follows that every projective-injective object in $\CP$ is isomorphic to $(P\st{1}\rt P)$ for some projective module $P$. We also have the following short exact sequence in $\CP$:
   $$\xymatrix@1{  0\ar[r] & {\left(\begin{smallmatrix} 0\\ P\end{smallmatrix}\right)}_{0}
			\ar[r]
			&  {\left(\begin{smallmatrix}P\\P\end{smallmatrix}\right)}_{1} \ar[r]&
			{\left(\begin{smallmatrix}P\\ 0\end{smallmatrix}\right)}_{0}\ar[r]& 0. } \ \    $$	
 Thus,  the exact category $\CP(\La)$ satisfies the condition $(iii)$ as well. Therefore, we have shown that $\CP(\La)$ is $0$-Auslander.
 \end{proof}

\begin{proposition}\label{Prop. injection}
    Let $\La$ and $\La'$ be artin algebras. If $\CP(\La)\simeq \CP(\La')$, as exact categories, then $\mmod \La\simeq \mmod \La'$
\end{proposition}
 \begin{proof}
     Let $F:\CP(\La)\rt \CP(\La')$ be an exact functor that gives an equivalence between the exact categories $\CP(\La)$ and $\CP(\La')$. Let $\CV$ and  $\CV'$ be the subcategories of  $\CP(\La)$ and  $\CP(\La')$, respectively, generated by injective objects.
     Since $F$ preserves injective objects, we have $F(\CV)=\CV'$. Consequently,  the equivalence $F$ induces an equivalence between quotient categories $\CP(\La)/\CV$ and $\CP(\La')/\CV'.$ Applying Theorem \ref{Thm-Cok-equiv},  we obtain an equivalence between $\mmod \La$ and $\mmod \La'.$
 \end{proof}

\subsection{The cokernel functor}

\begin{lemma}\label{Lemma-left-right-CP}
Let  $\left(\begin{smallmatrix} \phi_1\\ \phi_2\end{smallmatrix}\right):\left(\begin{smallmatrix} P_1\\ P_0\end{smallmatrix}\right)_{f}\rt \left(\begin{smallmatrix} Q_1\\ Q_0\end{smallmatrix}\right)_{g} $ be a morphism in $\CP$. Set $\phi=\left(\begin{smallmatrix} \phi_1\\ \phi_2\end{smallmatrix}\right)$, ${\rm P}=\left(\begin{smallmatrix} P_1\\ P_0\end{smallmatrix}\right)_f, {\rm Q}=\left(\begin{smallmatrix} Q_1\\ Q_0\end{smallmatrix}\right)_g$. Then
\begin{itemize}
    \item [$(1)$] If $\phi$ is right  almost split in $\CP$, then so is ${\rm Cok} (\phi):{\rm Cok}({\rm P})\rt {\rm Cok} ({\rm Q})$ in $\mmod \La$, provided ${\rm Cok}(\phi)$ is not a retraction;
    \item [$(2)$]If $\phi$ is left  almost split in $\CP$, then so is ${\rm Cok} (\phi):{\rm Cok}({\rm P})\rt {\rm Cok} ({\rm Q})$ in $\mmod \La,$
    provided ${\rm Cok}(\phi)$ is not a section.
\end{itemize}
\end{lemma}
\begin{proof}
We only prove $(1)$ and note that  the proof of $(2)$ is similar. Assume $h:M\rt {\rm Cok}({\rm Q})$ is not a retraction in $\mmod \La.$  We take a minimal projective presentation $P\st{h}\rt Q\st{\sigma}\rt  M\rt 0$. Then there exist  morphisms making the following diagram commutative:

	$$\xymatrix{		
		&	P\ar@{.>}[d]^{\beta_1} \ar[r]^{h} & Q
			\ar@{.>}[d]^{\beta_0} \ar[r]^{\sigma} & M \ar[d]^h\ar[r] &0 \\ 	&	Q_1\ar[r]^{g} & Q_0
			\ar[r] &  {\rm Cok}(Q) \ar[r]&0 }	$$
The above diagram gives us the morphism  $\beta=\left(\begin{smallmatrix} \beta_1\\ \beta_0\end{smallmatrix}\right):\left(\begin{smallmatrix} P\\ Q\end{smallmatrix}\right)_{h}\rt \left(\begin{smallmatrix} Q_1\\ Q_0\end{smallmatrix}\right)_{g}$. We claim that that $\beta=(\beta_1, \beta_0)$ is not a retraction. Otherwise, if there is $\alpha=\left(\begin{smallmatrix} \alpha_1\\ \alpha_0\end{smallmatrix}\right):\left(\begin{smallmatrix} Q_1\\ Q_0\end{smallmatrix}\right)_{g}\rt \left(\begin{smallmatrix} P\\ Q\end{smallmatrix}\right)_{h}$ with $\beta \alpha=1_{{\rm Q}}$, then ${\rm Cok}(\beta){\rm Cok}(\alpha)=h{\rm Cok}(\alpha)= {\rm Cok}(1_{{\rm Q}})=1_{{\rm Cok}(Q)}$, which is  a contradiction. Since $\phi$ is a right almost split morphism in $\CP$, the morphism $\beta$ factors through $\phi$, i.e, there exists $\gamma:(P\st{h}\rt Q)\rt (P_1\st{f}\rt P_0)$ such that $\phi \gamma=\beta$. Applying the cokernel functor on this factorization in $\CP$,  we obtain  the desired factorization in $\mmod \La.$
\end{proof}

\begin{remark}
Even if  the morphism $\phi:\left(\begin{smallmatrix} \phi_1\\ \phi_2\end{smallmatrix}\right)$ is minimal in $\CP$, it is not necessarily  true that the ${\rm Cok}(\phi)$ is  minimal in $\mmod \La$. We can see an example of this in  Lemma \ref{righ-left-minimimal-proj-obj}, where   the morphism $\left(\begin{smallmatrix} ip\\ 1 \end{smallmatrix}\right):\left(\begin{smallmatrix} Q\\ P\end{smallmatrix}\right)_{ip}\rt \left(\begin{smallmatrix} P\\ P\end{smallmatrix}\right)_1 $ is showen to be  minimal in $\CP$, but its cokernel ${\rm Cok}\left(\begin{smallmatrix} ip\\ 1 \end{smallmatrix}\right): P/{\rm rad}(P) \rt 0 $ is not minimal in $\mmod \La.$
\end{remark}
\begin{lemma}\label{righ-left-minimimal-proj-obj}
Let $P$ be an indecomposable projective module, and let $p:Q\rt {\rm rad}(P)$ be  a projective cover of ${\rm rad} (P)$, $i:{\rm rad}(P)\rt P$ be a canonical inclusion map. Then the following hold:
\begin{itemize}
    \item [$(1)$] The morphism $\left(\begin{smallmatrix} 0\\ ip \end{smallmatrix}\right):\left(\begin{smallmatrix} 0\\ Q\end{smallmatrix}\right)_0\rt \left(\begin{smallmatrix} 0\\ P\end{smallmatrix}\right)_0 $ is  right minimal almost split in $\CP$.
    \item [$(2)$] The morphism $\left(\begin{smallmatrix} ip\\ 1 \end{smallmatrix}\right):\left(\begin{smallmatrix} Q\\ P\end{smallmatrix}\right)_{ip}\rt \left(\begin{smallmatrix} P\\ P\end{smallmatrix}\right)_1 $ is  right minimal almost split in $\CP$. Moreover, the object $\left(\begin{smallmatrix} Q\\ P\end{smallmatrix}\right)_{ip}$
     is indecomposable in $\CP.$\\
\end{itemize}
\end{lemma}

\begin{proof}
We prove $(1)$ and  the proof of the first part of  $(2)$ is similar. Note that the morphism $(0, ip)$ can be expressed as  the following composition:
$$\left(\begin{smallmatrix} 0\\ ip \end{smallmatrix}\right):\left(\begin{smallmatrix} 0\\ Q\end{smallmatrix}\right)_0 \st{{\left(\begin{smallmatrix} 0\\ p\end{smallmatrix}\right)}}\rt \left(\begin{smallmatrix} 0\\ {\rm rad}(P)\end{smallmatrix}\right)_0 \st {\left(\begin{smallmatrix} 0\\ i\end{smallmatrix}\right)} \rt {\left(\begin{smallmatrix} 0\\ P\end{smallmatrix}\right)_0} $$
\noindent
By \cite[Lemma 3.1]{RS}, we know that the morphism ${\left(\begin{smallmatrix} 0\\ i\end{smallmatrix}\right)}$ is   minimal right almost split. Moreover, by Lemma \ref{righ-app-certan-obj}, the morphism ${\left(\begin{smallmatrix} 0\\ p\end{smallmatrix}\right)}$ is a minimal right $\CP$-approximation. Thus,  the composition of these two  morphisms is right minimal almost split  in $\CP$, which proves $(1)$.  For the second part of $(2)$,  assume that  there exists  a decomposition
 $$\left(\begin{smallmatrix} Q\\ P\end{smallmatrix}\right)_{ip}=\left(\begin{smallmatrix} Q'\\ P'\end{smallmatrix}\right)_{f'}\oplus \left(\begin{smallmatrix} Q''\\ P''\end{smallmatrix}\right)_{f''}.$$
 The decomposition implies that $Q=Q'\oplus Q''$ and $P=P'\oplus P''$. Since $P$ is indecomposable, either $P'=0$ or $P''=0$. Without of loss generality, assume $P''=0$. This implies that $Q''$ is a summand of $Q$ such that $p\mid_{Q''}=0$. This is not possible unless that $Q''=0$. Therefore, $\left(\begin{smallmatrix} Q\\ P\end{smallmatrix}\right)_{ip}$ is indecomposable in $\CP$, which completes the proof of the second part of $(2)$.
\end{proof}

Based on  the above lemma,  we can classify the  number of the  indecomposable direct summands of the right minimal almost split of an indecomposable projective object in $\CP$ into two types. Let $P$ be an indecomposable projective module. For  $(0 \rt P)$: the number is  equal to the number of direct summands of  ${\rm rad}(P)$. For  $(P\st{1}\rt P)$:  the number is always one.

\hspace{ 1 mm}

 The  following remark provides a useful way to consider short exact sequences  in $\CP.$
\begin{remark}\label{Remark-Split}
Since any short exact sequence of projective modules  splits,  it is not difficult to see that any short exact sequence in $\CP$  is isomorphic to a short exact sequence of the following form

	$$\xymatrix@1{  0\ar[r] & {\left(\begin{smallmatrix} P_1\\ P_2\end{smallmatrix}\right)}_{f}
			\ar[rr]^-{\left(\begin{smallmatrix} [1~~0]^t \\ [1~~0]^t\end{smallmatrix}\right)}
			& & {\left(\begin{smallmatrix}P_1\oplus Q_1\\ P_2\oplus Q_2\end{smallmatrix}\right)}_{h}\ar[rr]^-{\left(\begin{smallmatrix} [0~~1] \\ [0~~1]\end{smallmatrix}\right)}& &
			{\left(\begin{smallmatrix}Q_1\\ Q_2\end{smallmatrix}\right)}_{g}\ar[r]& 0, } \ \    $$
where ${h}= \tiny {\left[\begin{array}{ll} f & q \\ 0 & g \end{array} \right]}$ with $q:Q_1 \rt P_2$.
 \end{remark}

\begin{proposition}\label{Cok-preserve-almost-split}
Let $\eta: \ 0 \rt {\rm X}\st{\phi}\rt  {\rm Y}\st{\psi}\rt {\rm Z}\rt 0$ be an almost split sequence in $\CP$. If  ${\rm Z}$ is not isomorphic to either $(P\rt 0)$,  for some indecomposable projective module $P$,
 then
$${\rm Cok}(\eta): \ 0 \rt {\rm Cok}({\rm X})\rt {\rm Cok}({\rm Y})\rt {\rm Cok}({\rm Z})\rt 0$$
is an almost split sequence in $\mmod \La$.

\end{proposition}
\begin{proof}
Let ${\rm X}=(P_1\st{f}\rt P_2), \ {\rm Z}=(Q_1\st{h}\rt Q_2)$. Using Remark \ref{Remark-Split}, we can write the short exact sequence $\eta$ as:
$$\xymatrix@1{  0\ar[r] & {\left(\begin{smallmatrix} P_1\\ P_2\end{smallmatrix}\right)}_{f}
			\ar[rr]
			& & {\left(\begin{smallmatrix} P_1\oplus Q_1\\ P_2\oplus Q_2\end{smallmatrix}\right)}_{h}\ar[rr]& &
			{\left(\begin{smallmatrix}Q_1 \\ Q_2\end{smallmatrix}\right)}_{g}\ar[r]& 0. } \ \    $$
\noindent Denote ${\rm Y}=(P_1\oplus Q_1\st{h}\rt P_2\oplus Q_2).$ The snake lemma yields the following commutative diagram in $\mmod \La$
$$\xymatrix{& 0 \ar[d] & 0 \ar[d] & 0 \ar[d]\\
		0 \ar[r] & K_1 \ar[d] \ar[r] &  K_2 \ar[d]
		\ar[r]^{\lambda} & K_3 \ar[d]^i&\\
		0 \ar[r] &P_1\ar[d]^{f} \ar[r] & P_1\oplus Q_1
		\ar[d]^{h}\ar[r] & Q_1 \ar[d]^{g} \ar[r] & 0\\
		0 \ar[r] &P_2 \ar[d] \ar[r] & P_2\oplus Q_2
		\ar[d] \ar[r] & Q_2 \ar[d] \ar[r] & 0\\
		& {\rm Cok}(\mathrm{X}) \ar[r]^{{\rm Cok}(\phi)}\ar[d] & {\rm Cok}(\mathrm{Y}) \ar[r]^{{\rm Cok}(\psi)}\ar[d]  & {\rm Cok}(\mathrm{Z}) \ar[r]\ar[d] &0\\ & 0  & 0  & 0 &  }
	$$
We claim that $\la$ is surjective.  To prove this, let $Q\st{p}\rt K_3\rt $ be a projective cover. Then, there is a morphism $(ip, 0):(Q\rt 0)\rt (Q_1\st{g}\rt Q_2)$ in $\CP$. By our assumption, the morphism $(ip, 0)$ is clearly a   non-retraction. Since $\eta$ is an almost split sequence in $\CP$, there exists a morphism $(s, 0):(Q\rt 0)\rt (P_1\oplus Q_1\st{h}\rt P_2\oplus Q_2)$ with $\psi (s,~ 0)=(ip, 0).$ The commutative diagram induced by the morphism $(s,~ 0)$ implies that ${\rm Im}(s)\subseteq K_2.$ This implies that $\la$ is surjective, as required.  Therefore,  the image of $\eta$ under the cokernel functor is  a short exact sequence in $\mmod \La$. It is impossible for  the short exact sequence ${\rm Cok}(\eta)$ to  split; Otherwise, it would  follow that $\eta$ splits, which  is absurd. Finally, Lemma \ref{Lemma-left-right-CP} shows that ${\rm Cok}(\eta)$ is an almost split sequence in $\mmod \La.$ This completes the proof.
\end{proof}

\begin{remark}
In \cite{EH}, a functor $\Theta_{\CC}:{\rm H}(\CC)\rightarrow \mmod \CC$ is introduced for a dualizing variety $\CC$. It is shown  that this functor behaves well with respect to the Auslander-Reiten theory. Specifically, it is demonstrated that the cokernel functor ${\rm Cok}$ can be viewed as a special case of this general observation.
\end{remark}

\subsection{From $\Gamma_{\La}$ to $\Gamma_{\CP}$}\label{FromArtoCP}
According to Theorem \ref{Thm-Cok-equiv} and Proposition \ref{Cok-preserve-almost-split}, we can identify $\Gamma_{\La}$ as a full subquiver of $\Gamma_{\CP}:=\Ga_{\CP(\La)}$. This embedding is given by sending a vertex corresponding to  an  indecomposable module $M$  to the vertex in $\Ga_{\CP}$ corresponding to the minimal projective presentation of $M$. Conversely, we can obtain $\Ga_{\CP}$ from $\Gamma_{\La}$ by adding vertices corresponding to indecomposable objects of the forms $(P\st{1}\rt P)$ and $(P\rt 0)$ and relevant arrows. To add the missing vertices and arrows one can follow the following steps: (remember that throughout this algorithm and later as well,  we identify the vertices with iso-classes of indecomposable modules or objects, and indecomposable modules with their minimal projective presentations.)

Let $P$ be an indecomposable projective module.
\begin{enumerate}

\item  $(P\st{1}\rt P)$: in view of Lemma \ref{righ-left-minimimal-proj-obj}, we need to consider  two cases:
 \begin{itemize}
 \item [$(i)$] if $P/{\rm rad}(P)$ is not injective, then   we add $(P\st{1}\rt P)$  to the middle  of  the mesh starting from $P/{\rm rad}(P)$, and we add an arrow  from $P/{\rm rad}(P)$ to $(P\st{1}\rt P)$, as well as  an arrow form $(P\st{1}\rt P)$ to $\tau^{-1}P/{\rm rad}(P)$.
     \item [$(ii)$] if $P/{\rm rad}(P)$ is  injective, then  we add $(P\st{1}\rt P)$ along with an arrow from $P/{\rm rad}(P)$ to $(P\st{1}\rt P)$.
     \end{itemize}
\item   $(P\rt 0)$: in view of Proposition \ref{minimal-appro-projective-dom},   we add arrows  properly from the direct summands of $\nu P/{\rm soc}(\nu P) $ to  the vertex $(P \rt 0)$,  and  add arrows properly from the vertex  $(P\rt 0)$ to the inverse Auslander-Reiten translation of non-injective direct summands of $\nu P/{\rm soc}(\nu P)$.
\end{enumerate}

Overall, this construction allows us to understand the relationship between $\Gamma_{\La}$ and $\Gamma_{\CP}$, and provides a way to obtain $\Gamma_{\CP}$ from $\Gamma_{\La}$ by adding certain vertices and arrows.   Our result in next section are helpful for this construction.

It is worth noting this close relationship between $\Gamma_{\La}$ and $\Gamma_{\CP}$ has been  investigated in  earlier works, e.g. see \cite[Proposition 4.12]{CPS}.
\subsection{The morphism category of injective modules}\label{injective-morphism}
Let $\CI:=\CI(\La)$ denote the subcategory of $\CH$ consisting of all morphisms with injective domain and injective codomain. We  can apply  a similar  consideration   to $\CI$ as  we have already  done for $\CP$. We present  the similar (important) results on $\CI(\La)$ without their proofs, as the proofs are similar.
\begin{theorem}\label{Theorem-morphism-functor-injective}
The functor ${\rm Ker}:\CI(\La) \lrt \mmod \La$ is full, dense, and objective.  Thus, the kernel functor  makes the following diagram commute.
\[\xymatrix{
\CI(\La)\ar[r]^{\rm Ker}\ar[d]^{\pi} & \mmod \La\\
\frac{\CI(\La)}{\CU}\ar[ur]_{\overline{\rm Ker}}   \\
}\]
\noindent
Here, $\pi$ is the natural quotient map, and $\CU$ is the full subcategory of $\CI(\La)$ generated by all finite direct sums of objects of  the forms $(I\st{1}\rt I)$ or $(0\rt I)$, where  $I$ is an injective module.
\end{theorem}

\begin{proposition}
The subcategory $\CI$ is functorially finite in $\CH.$ In particular, $\CI$ has almost split sequences.
\end{proposition}

\begin{proposition}\label{Cok-ass}
Let $\eta: \ 0 \rt {\rm X}\st{\phi}\rt  {\rm Y}\st{\psi}\rt {\rm Z}\rt 0$ be an almost split sequence in $\CI$. If ${\rm X}$ is not isomorphic to  $(0\rt I)$,  for some indecomposable injective module $I$, then
$${\rm Ker}(\eta): \ 0 \rt {\rm Ker}({\rm X})\rt {\rm Ker}({\rm Y})\rt {\rm Ker}({\rm Z})\rt 0$$
is an almost split sequence in $\mmod \La$.
\end{proposition}
The Nakayama functor $\nu$ induces an equivalence between $\CP(\La)$ and $\CI(\La)$. More precisely, it is defined by sending $(P\st{f}\rt Q) \in \CP(\La)$ to $(\nu P\st{\nu f} \rt \nu Q) \in \CI(\La)$. The induced equivalence on the morphism category  is still denoted  by $\nu$.

The functor $(-)^*=\Hom_{\La}(-, \La)$ in general does not give  a duality between $\mmod \La$ and $\mmod \La^{\rm op}$. It always provides a duality between ${\rm prj}\mbox{-}\La$ and ${\rm prj}\mbox{-}\La^{\rm op}$. But this duality induces a duality between the category of morphisms of projective modules. More precisely, an object $(P\st{f}\rt Q) \in \CP(\La)$ is mapped to $(Q^*\st{f^*}\rt P^*)$ in $\CP(\La^{\rm op})$. We still denote by $(-)^*$ the induced duality. Furthermore, we have the following  commutative diagram:
 \[\xymatrix{\CP(\La)\ar[d]^{(-)^*} \ar[r]^{\rm Cok}& \mmod \La \ar[d]^{(-)^*} \\ \CP(\La^{\rm op})  \ar[r]^{\rm Ker} & \mmod \La^{\rm op}   }\]

 In the sequel,  we will observe that the duality $(-)^*:\CP(\La)\rt \CP(\La^{\rm op}) $ is helpful to state the dual version of our results.

There are several  consequences for the module category  can be proved   by combining the equivalences on the   morphism categories. Although these results may be known  in the literature,  proving them  using the morphism category approach is still  interesting.

Let us start by recovering the equivalence between the stable categories defined by  the Auslander-Reiten translation.

For this, let $\CW=<\CV \cup \{(0 \rt P)\mid P \in {\rm prj}\mbox{-}\La\}>$, and let $\CW'=<\CU\cup  \{(I\rt 0)\mid I \in {\rm inj}\mbox{-}\La\}>$. It is easy to see that the equivalence given in Theorem \ref{Thm-Cok-equiv} can be extended to $\CP(\La)/\CW\simeq \underline{\rm mod}\mbox{-}\La$, denote by ${\overline{\rm Cok}}_p$,  and similarly, by Theorem \ref{Theorem-morphism-functor-injective}, there is the equivalence $\CI(\La)/ \CW' \simeq \overline{\rm mod}\mbox{-}\La$, denoted by $\overline{{\rm Ker}}_I$. In addition, since $\nu (\CW)=\CW'$, we have the equivalence $\overline{\nu}=\CP(\La)/\CW\simeq \CI(\La)/\CW'$. Altogether, we obtain  the promised  equivalence

$$\theta: \underline{\rm mod}\mbox{-}\La \st{({\overline{\rm Cok}})^{-1}_p}\lrt \CP(\La)/\CW \st{\overline{\nu}}\rt \CI(\La)/\CW'\st{\overline{{\rm Ker}}_I}\lrt \overline{\rm mod}\mbox{-}\La.  $$

By following the definition, we observe that for any indecomposable non-projective module $M$, $\theta(M)=\tau_{\La}M.$\\
We also note  that one can obtain the transpose ${\rm Tr}:\mmod \La\rt \mmod \La^{\rm op}$ by taking the duality $(-)^*:\CP(\La)\rt \CP(\La^{\rm op})$ along with the cokernel functors. Namely,
\begin{equation}\label{Tran-Eq}
     \underline{{\rm mod}}\mbox{-} \La \st{\overline{\rm Cok}^{-1}_{\La}} \rt \CP(\La)/\mathcal{Y}\st{(-)^*}\rt \CP(\La^{\rm op})/\mathcal{Y}'\st{\overline{\rm Cok}_{\La^{\rm op}}}\rt \underline{{\rm mod}}\mbox{-} \La^{\rm op}
   \end{equation}
where $\CY=<(0\rt P), (P\st{1}\rt P), (P\rt 0)| \ P \in {\rm prj}\mbox{-}\La> $, and   $\CY'$ is defined similarly.\\

\section{certain almost split sequences in $\CP(\La)$}\label{Section6}
In this section the almost split sequences in $\CP$ ending  at (or starting from) objects of the form either $(0 \rt P)$ or $(P\rt 0)$ will be determined explicitly. These results will be helpful  in constructing  the Auslander-Reiten quiver of $\CP$, particularly when  the  knitting procedure works well. As an  application,  we will  provide additional  information for the canonical map introduced in Remark \ref{canonicalRemark}.
\subsection{the object $(P \rt 0)$}

\begin{lemma}
 Let $M$ be a  module and $P_1\st{f}\rt P_0\st{\sigma}\rt M\rt 0$ a minimal projective presentation of $M$. Then the object  $(P_1\st{f}\rt P_0)$  in $\CP$ is indecomposable if and only if $M$ is indecomposable in $\mmod \La.$
\end{lemma}
\begin{proof}
Assume $(P_1\st{f}\rt P_0)$ is an indcomposable object in $\CP$. If $M$ decomposes as $M_1\oplus M_2$, then let $Q^1_1\st{f^1}\rt Q^1_0\rt M_1\rt 0$ and $Q^2_1\st{f^2}\rt Q^2_0\rt M_2\rt 0$ be minimal projective presentations of $M_1$ and $M_2$, respectively. Then the direct sum of two minimal projective presentations provides a minimal projective presentation of $M$. Because of the uniqueness of minimal projective presentation up to isomprphism, we can deduce an isomorphism $(P_1\st{f}\rt P_0)\simeq (Q_1^1\st{f^1}\rt Q_0^1)\oplus (Q_1^2\st{f^2}\rt Q_0^2)$ in $\CP$, which  contradicts the assumption that  $(P_1\st{f}\rt P_0)$ is indecomposable.  Conversely, assume that $M$ is an indecomposable module. If there is an decomposition  $$(P_1\st{f}\rt P_0)= (Q_1^1\st{f^1}\rt Q_0^1)\oplus (Q_1^2\st{f^2}\rt Q_0^2),$$
then by applying  the cokernel functor we get

$$M={\rm Cok}(f^1)\oplus {\rm Cok}(f^2).$$
But since $M$ is indecomposable, we have either ${\rm Cok}(f^1)=0$ or ${\rm Cok}(f^2)=0$. Without loss of generality,  assume that ${\rm Cok}(f^1)=0$. Then  $(Q^1_1\st{f^1}\rt Q^1_0)$ is isomorphic to a direct sum of objects of the form $(P\rt 0)$ or $(P\st{1}\rt P)$. This  contradicts the minimality of the projective presentation $P_1\st{f}\rt P_0\rt M\rt 0$ of $M$. Therefore, $(P_1 \st{f}\rt P_0)$ must be indecomposable in $\CP.$
\end{proof}

Recall that $\alpha_P$ is defined in Remark \ref{canonicalRemark}.

\begin{proposition}\label{minimal-appro-projective-dom}
Let $P$ be an indecomposable projective module. Let $P_1\st{g}\rt P_0 \st{\sigma}\rt  \nu P \rt 0 $ be a minimal projective presentation of $\nu P$, and $e:P\rt P_0$ a lifting of the morphism $\alpha_P$ through $\sigma.$ Then the following sequence is an almost split sequence in $\CP(\La)$
	$$\xymatrix@1{  0\ar[r] & {\left(\begin{smallmatrix} P_1\\ P_0\end{smallmatrix}\right)}_{g}
			\ar[rr]^-{\left(\begin{smallmatrix} [1~~0]^t\\ 1\end{smallmatrix}\right)}
			& & {\left(\begin{smallmatrix}P_1\oplus P\\  P_0 \end{smallmatrix}\right)}_{[g~~e]}\ar[rr]^-{\left(\begin{smallmatrix} [0~~1] \\ 0\end{smallmatrix}\right)}& &
		{\left(\begin{smallmatrix}P\\0\end{smallmatrix}\right)}_{0}\ar[r]& 0 }.$$
		
	\end{proposition}

\begin{proof}
By Proposition \ref{proj-nakaya},  there is the following almost split sequence in $\CH$ ending at $(P \rt 0)$
	$$\xymatrix@1{  0\ar[r] & {\left(\begin{smallmatrix} 0\\ \nu P\end{smallmatrix}\right)}_{0}
			\ar[rr]^-{\left(\begin{smallmatrix} 0\\ 1\end{smallmatrix}\right)}
			& & {\left(\begin{smallmatrix}P\\ \nu P \end{smallmatrix}\right)}_{\alpha_P}\ar[rr]^-{\left(\begin{smallmatrix} 1 \\ 0\end{smallmatrix}\right)}& &
			{\left(\begin{smallmatrix}P\\0\end{smallmatrix}\right)}_{0}\ar[r]& 0. }$$
  Applying Proposition \ref{Prop.dom-proj.approx},  we obtain the following minimal right $\CP$-approximations of $(P\st{\alpha_{P}} \rt \nu P)$ and $(0 \rt \nu P)$, respectively:			
 $$\left(\begin{smallmatrix} [0~~e]\\ \sigma \end{smallmatrix}\right):\left(\begin{smallmatrix} P_1\oplus P\\ P_0\end{smallmatrix}\right)_{[g~~e]}\rt \left(\begin{smallmatrix} P\\ \nu P\end{smallmatrix}\right)_{\alpha_P} $$			
where $e:P\rt P_0$ is a lifting of $f$ through $\sigma,$ and
$$\left(\begin{smallmatrix} 0\\ \sigma \end{smallmatrix}\right):\left(\begin{smallmatrix} P_1\\ P_0\end{smallmatrix}\right)_g\rt \left(\begin{smallmatrix} 0\\ \nu P\end{smallmatrix}\right)_{0} $$	
Now we can build the following pull-back diagram in $\CH$:

	$$\xymatrix{		
		 \ \ 0\ar[r]&{\left(\begin{smallmatrix} P_1\\ P_0\end{smallmatrix}\right)_{g}}	\ar[d]^{{\left(\begin{smallmatrix} 0\\ \sigma \end{smallmatrix}\right)}} \ar[r] & {\left(\begin{smallmatrix} P_1\oplus P\\ P_0\end{smallmatrix}\right)_{[g~~e]}}
			\ar[d]^{{\left(\begin{smallmatrix} [0~~e]\\ \sigma \end{smallmatrix}\right)}} \ar[r] & {\left(\begin{smallmatrix} P\\  0\end{smallmatrix}\right)_{0} } \ar@{=}[d]\ar[r] &0 \\ \ \		0 \ar[r]&	{\left(\begin{smallmatrix} 0\\ \nu P\end{smallmatrix}\right)_{0} } \ar[r] & {\left(\begin{smallmatrix} P\\ \nu P\end{smallmatrix}\right)_{\alpha_P} }
			\ar[r] & {\left(\begin{smallmatrix} P\\ 0\end{smallmatrix}\right)_{0} } \ar[r]&0 }	$$
			Therefore, by Theorem \ref{KP-ext-proj}, the top row is an almost split sequence in $\CP$, as desired.
   \end{proof}

The above result is helpful to provide more information about the canonical morphism $\alpha_P.$

\begin{corollary}\label{CorollaryCokernel}
Let $P$ be an indecomposable projective module. Then there is following exact sequence in $\mmod \La$

$$0 \rt \Omega^2_{\La}(\nu P)\rt \Omega^2_{\La}(\nu P/{\rm soc} (\nu P))\oplus Q\rt  P \st{\alpha_P}\rt \nu P  \rt \nu P/{\rm soc} (\nu P)\rt  0, $$
where  $Q$ is a projective module.
\end{corollary}
\begin{proof}
We begin by considering the almost split sequence ending at $(P\rt 0)$ given in Proposition \ref{minimal-appro-projective-dom}, which we denote as $\delta$. This sequence  induces the following commutative diagram:
$$\xymatrix{		
		\ \ 0\ar[r]&	P_1\ar[d]^g \ar[r]^{[1~~0]^t} & P_1\oplus P
			\ar[d]^{[g~~e]} \ar[r]^{[0~~1]} & P \ar[d]\ar[r] &0 \\  \ \		0 \ar[r]&	P_0\ar[r]^1 & P_0
			\ar[r] &  0\ar[r]&0 }	$$
We know that the canonical quotient map $\nu P\rt \nu P/{\rm soc} (\nu P)$ is a minimal left almost split morphism in $\mmod \La$ starting from $\nu P.$ On the other hand, Lemma \ref{Lemma-left-right-CP} implies that applying cokernel functor to  the monomorphism lying in the almost split sequence $\delta$  gives us another  minimal  left almost split morphism starting at $\nu P$. By  uniqueness, the cokernel of the middle object in $\delta$ is isomorphic to $\nu P/{\rm soc}(\nu P)$. Combining  facts,  we obtain the following diagram:

$$\xymatrix{& 0 \ar[d] & 0 \ar[d] & 0 \ar[d]\\
		0 \ar[r] & \Omega_{\La}(\nu P ) \ar[d] \ar[r] & \Omega^2_{\La}(\nu P/{\rm soc}(\nu P))\oplus Q  \ar[d]
		\ar[r] & P \ar@{=}[d]&\\
		0 \ar[r] &P_1\ar[d]^{g} \ar[r] & P_1\oplus P
		\ar[d]^{[g~~e]}\ar[r] & P \ar[d] \ar[r] & 0\\
		0 \ar[r] &P_0 \ar[d] \ar[r] & P_0
		\ar[d] \ar[r] & 0 \ar[d] \ar[r] & 0\\
		& \nu P \ar[r]\ar[d] & \nu P/{\rm soc } (\nu P) \ar[d]\ar[r]  & 0  &\\ & 0  & 0  &  &  }
	$$
Note that, since the object $(P_1\oplus P\st{[g~~e]}\rt P)$  might not induce a minimal projective presentation,  we may have a projective summand $Q$ in the second row. To complete the proof, we use the snake lemma.
\end{proof}

 It is not generally true that the projective module $Q$ can be removed from  the exact sequence given in Corollary \ref{CorollaryCokernel}. In the following example,  we show that   $Q$ can not be removed  even when  $P$ is projective-injective.
\begin{example}
Let   $\La$ be  given by  the quiver $A_3:3\st{\alpha}\rt 2 \st{\beta}  \rt 1$ and relation $\alpha \beta=0$. Let  $P={\bsm 3\\ 2\esm}$, which is projective-injective. The exact sequence in Corollary \ref{CorollaryCokernel} is given by:
$$ 0 \rt 1 \rt 0 \oplus {\bsm 2\\ 1\esm} \rt {\bsm 3\\ 2\esm} \rt 3\rt 0 $$
Here $Q={\bsm 2\\ 1\esm}\neq 0$.
\end{example}

\subsection{indecomposable projective-injective modules}

\begin{proposition}\label{proj-inj-indecom}
	 Let $P$ be an indecomposable  projective-injective module. Then, there exists an indecomposable projective module $Q$ such that:
\begin{itemize}
\item [$(i)$] $P\simeq \nu Q$;
\item [$(ii)$] ${\rm top}(Q) \simeq {\rm soc}(P) $;
\item [$(iii)$] the almost split sequence in $\CP$ starting from $(0 \rt P)$ is of the  form
		$$\xymatrix@1{ \eta: \ \ 0\ar[r] & {\left(\begin{smallmatrix} 0\\ P\end{smallmatrix}\right)}_{0}
			\ar[r]
			&  {\left(\begin{smallmatrix}Q\\ \nu Q \end{smallmatrix}\right)}_{\alpha_Q}\ar[r]&
			{\left(\begin{smallmatrix}Q\\0\end{smallmatrix}\right)}_{0}\ar[r]& 0 }.$$
			\end{itemize}
\end{proposition}
\begin{proof}
Since $P$ is an injective module and the Nakayama functor $\nu:{\rm prj}\mbox{-}\La\rt {\rm inj}\mbox{-}\La$ is an equivalence, there exists  an indecomposable projective module $Q$ such that $\nu Q\simeq P.$  Proposition \ref{proj-nakaya} gives us the  following almost split sequence in $\CH$:

$$\xymatrix@1{  0\ar[r] & {\left(\begin{smallmatrix} 0\\ \nu Q\end{smallmatrix}\right)}_{0}
			\ar[r]^-{\left(\begin{smallmatrix} 0\\ 1\end{smallmatrix}\right)}
			&  {\left(\begin{smallmatrix}Q\\ \nu Q \end{smallmatrix}\right)}_{\alpha_Q}\ar[r]^-{\left(\begin{smallmatrix} 1 \\ 0\end{smallmatrix}\right)}&
			{\left(\begin{smallmatrix}Q\\0\end{smallmatrix}\right)}_{0}\ar[r]& 0. }$$
Then by considering the isomorphism  $P\simeq \nu Q$, we obtain the desired almost split sequence in $\CP$.  The remaining statement $(ii)$ follows from the observation given in Remark \ref{canonicalRemark}.
\end{proof}

 When $P$ in the above proposition is also a simple module, then ${\rm soc}(P)=P$. Hence, there is the following almost split sequence in $\CP$
$$\xymatrix@1{  0\ar[r] & {\left(\begin{smallmatrix} 0\\ P\end{smallmatrix}\right)}_{0}
			\ar[r]^-{\left(\begin{smallmatrix} 0\\ 1\end{smallmatrix}\right)}
			&  {\left(\begin{smallmatrix}P\\ P\end{smallmatrix}\right)}_{1}\ar[r]^-{\left(\begin{smallmatrix} 1 \\ 0\end{smallmatrix}\right)}&
			{\left(\begin{smallmatrix}P\\0\end{smallmatrix}\right)}_{0}\ar[r]& 0. }$$

 The existence of such a simple projective-injective module $P$ is  is equivalent to saying that $\La$ has  a simple algebra as a direct summand.
	


\begin{proposition}\label{Prop-almost1}
Keep in mind the notation in Proposition \ref{proj-inj-indecom}.  There is the following almost split sequence in $\CP$
	$$\xymatrix@1{  \ \ 0\ar[r] & {\left(\begin{smallmatrix} Q'\\ P'\end{smallmatrix}\right)}_{f}
			\ar[r]
			&  {\left(\begin{smallmatrix}Q'\oplus Q\\ P' \end{smallmatrix}\right)}_{h}\oplus {\left(\begin{smallmatrix}0\\P\end{smallmatrix}\right)}_{0} \ar[r]&
			{\left(\begin{smallmatrix}Q\\ \nu Q\end{smallmatrix}\right)}_{\alpha_Q}\ar[r]& 0 },$$
where ${h}=[f~~q] $ with a non-zero morphism $q:Q \rt P'$. 		
\end{proposition}
\begin{proof}
Set $\tau_{\CP}(Q\st{\alpha_Q}\rt \nu Q)=(Q'\st{f}\rt P')$. From the following almost split sequence given in   Proposition \ref{proj-inj-indecom},
$$\xymatrix@1{  0\ar[r] & {\left(\begin{smallmatrix} 0\\ P\end{smallmatrix}\right)}_{0}
			\ar[r]^-{\left(\begin{smallmatrix} 0\\ 1\end{smallmatrix}\right)}
			&  {\left(\begin{smallmatrix}Q\\ \nu Q \end{smallmatrix}\right)}_{\alpha_Q}\ar[r]^-{\left(\begin{smallmatrix} 1 \\ 0\end{smallmatrix}\right)}&
			{\left(\begin{smallmatrix}Q\\0\end{smallmatrix}\right)}_{0}\ar[r]& 0 },$$
we observe that $(0 \rt P)$ lies in the middle term of the almost split sequence ending at $(Q \st{\alpha_Q}\rt \nu Q).$ In view of Remark \ref{Remark-Split} and considering the isomorphism $P\simeq \nu Q$,  the almost split sequence  with ending $(Q\st{\alpha_Q}\rt \nu Q)$ has the required form.
\end{proof}

\begin{remark}\label{Rem-rad-soc}
    From \cite[Proposition 8.6]{SY}, we know that if $P$ is an indecomposable projective-injective module in $\mmod \La$ of length at least two, then there is an almost  split sequence in $\mmod \La$ of the following form
    \begin{equation}  \label{a.s.s.rad}
    0 \rt {\rm rad}(P) \rt {\rm rad}(P)/{\rm soc}
    (P)\oplus P\rt P/{\rm soc}(P)\rt 0 \end{equation}
Keep in mind the notations used in Proposition \ref{Prop-almost1}. Let   $Q'\st{f}\rt P'\rt {\rm rad}(P)\rt 0$ be  a minimal projective presentation. We  obtain that $Q'\oplus Q\st{h}\rt P'\rt {\rm rad}(P)/{\rm soc}(P)\rt 0$ is a minimal projective presentation.

\end{remark}

\subsection{ the object $(0 \rt P)$}
 Consider the duality $(-)^*:\CP(\La)\rt \CP(\La^{\rm op})$, as defined in  \ref{injective-morphism}.
 The duality clearly preserves the almost split sequences, i.e., if $\la: 0 \rt {\rm X} \st{f}\rt {\rm Y} \st{g}\rt {\rm Z}\rt 0$ is an almost split sequence in $\CP(\La)$, then  $\la^*: 0 \rt{\rm Z}^* \st{g^*}\rt {\rm Y}^* \st{f^*}\rt {\rm X^*}\rt 0 $ is an almost split sequence in $\CP(\La^{\rm op})$.

\begin{proposition}\label{starting0P}
    Let $P$ be an indecomposable projective module. Then, the almost split sequence starting at $(0 \rt P)$ has the following form
$$\xymatrix@1{  0\ar[r] & {\left(\begin{smallmatrix} 0\\ P\end{smallmatrix}\right)}_{0}
			\ar[rr]
			& & {\left(\begin{smallmatrix}P_0^*\\  P_1^*\oplus P^*\end{smallmatrix}\right)}_{[g^*~~e^*]^t}\ar[rr]& &
		{\left(\begin{smallmatrix}P^*_0\\P^*_1\end{smallmatrix}\right)}_{g^*}\ar[r]& 0 },$$
  where $P_1\st{g}\rt P_0 \rt \nu_{\La^{\rm op}} P^*\rt 0$ is a minimal projective presentation in $\mmod \La^{\rm op}.$
\end{proposition}
\begin{proof}
By applying Proposition \ref{minimal-appro-projective-dom} to the projective module $P^*$, we obtain the following almost split sequence in $\CP(\La^{\rm op})$
   	$$\xymatrix@1{  0\ar[r] & {\left(\begin{smallmatrix} P_1\\ P_0\end{smallmatrix}\right)}_{g}
			\ar[rr]^-{\left(\begin{smallmatrix} [1~~0]^t\\ 1\end{smallmatrix}\right)}
			& & {\left(\begin{smallmatrix}P_1\oplus P\\  P_0 \end{smallmatrix}\right)}_{[g~~e]}\ar[rr]^-{\left(\begin{smallmatrix} [0~~1] \\ 0\end{smallmatrix}\right)}& &
		{\left(\begin{smallmatrix}P^*\\0\end{smallmatrix}\right)}_{0}\ar[r]& 0 }.$$
  Applying  the duality $(-)^*$ to the almost split sequence and using the  isomorphism $P^{**}\simeq P$, we obtain the desired almost split sequence.
\end{proof}

When $P$ is an indecomposable non-injective module, there is an alternative method  to compute the almost spit sequence in $\CP$ starting at $(0 \rt P)$.

\begin{proposition}\label{minimal-appro-projective-dual}
Let $P$ be an indecomposable non-injective projective module.
  Let
 $ 0 \rt P\st{u}\rt B\st{v} \rt \tau^{-1}_{\La} P\rt 0$ be an almost split sequence  in $\mmod \La$. Consider the following commutative pull-back diagram
	
		$$\xymatrix{		   & \Omega_{\La}(\tau^{-1}_{\La} P) \ar[d]^{h}
		\ar@{=}[r] & \Omega_{\La}(\tau^{-1}_{\La}P) \ar[d]^{i}  \\
		 \ar@{=}[d] \ar[r]^{\left[\begin{smallmatrix} 1\\ 0\end{smallmatrix}\right]} P& P\oplus Q_0
		\ar[d]^d \ar[r]^{\left[\begin{smallmatrix} 0 & 1\end{smallmatrix}\right]} & Q_0\ar[d]  \\
		 P \ar[r]^{u} & B
		 \ar[r]^{v} & \tau^{-1}_{\La}P }	$$
		Then, the following short exact sequence is an almost split sequence in $\CH$,
		$$\xymatrix@1{ \ \ 0\ar[r] & {\left(\begin{smallmatrix} 0\\ P\end{smallmatrix}\right)}_{0}
			\ar[rr]^-{\left(\begin{smallmatrix} 0\\ \left[\begin{smallmatrix} 1\\ 0\end{smallmatrix}\right]\end{smallmatrix}\right)}
			& & {\left(\begin{smallmatrix}Q_1\\ P\oplus Q_0 \end{smallmatrix}\right)}_{hj}\ar[rr]^-{\left(\begin{smallmatrix} 1 \\ \left[\begin{smallmatrix} 0 & 1\end{smallmatrix}\right]\end{smallmatrix}\right)}& &
			{\left(\begin{smallmatrix}Q_1\\Q_0\end{smallmatrix}\right)}_{ij}\ar[r]& 0 },$$
		
\noindent		
		 where $Q_1\st{j}\rt \Omega_{\La}(\tau^{-1}_{\La} P)\rt 0 $ a projective cover.   In particular, it is an almost split sequence in $\CP.$

	\end{proposition}
\begin{proof}
     It follows from Proposition \ref{Reminder:0M}.
\end{proof}

\section{Some application: $g$-vectors}

In this section, we collect some  consequence for $\mmod \La$ by making use  of our study of the morphism category $\CP$  established in the preceding sections. While some applications  can  be proved directly, it is interesting to observe them in terms of the morphism category. Lastly, we will provide some results about $g$-vectors, which are important in $\tau$-tilting theory as introduced in \cite{AIR}.

\hspace{ 1 mm}

We begin by presenting  two immediate consequences of our results:
\begin{proposition}
    Let $P$ be an indecomposable projective module.    Then, the following assertions hold.
    \begin{itemize}
   \item  [$(1)$] If $P$ is injective, then $P\simeq \nu_{\La}(\nu_{\La^{\rm op}}P^*)^*$;

 \item  [$(2)$] If $P$ is non-injective, then  there is the following exact sequence in $\mmod \La,$
$$ 0\rt  (\nu_{\La^{\rm op}}P^*)^*\rt P^*_0\st{g^*}\rt P_1^*\rt \tau^{-1}_{\La}P\rt 0,$$     where  $ P_1\st{g}\rt P_0\rt \nu_{\La^{op}} P^*\rt 0 $ is  a minimal projective presentation in $\mmod \La^{\rm op}$.
     \end{itemize}
\end{proposition}
\begin{proof}
$(1)$    According to Proposition \ref{proj-inj-indecom}, there is the following  almost split sequence in $\CP(\La)$

$$\xymatrix@1{  \ \ 0\ar[r] & {\left(\begin{smallmatrix} 0\\ P\end{smallmatrix}\right)}_{0}
			\ar[rr]^-{\left(\begin{smallmatrix} 0\\ 1\end{smallmatrix}\right)}
			& & {\left(\begin{smallmatrix}Q\\ P \end{smallmatrix}\right)}_{f}\ar[rr]^-{\left(\begin{smallmatrix} 1 \\ 0\end{smallmatrix}\right)}& &
			{\left(\begin{smallmatrix}Q\\0\end{smallmatrix}\right)}_{0}\ar[r]& 0 }$$
   with $\nu_{\La}Q=P.$ Now, by Proposition \ref{starting0P}, $Q\simeq (\nu_{\La^{\rm op}}P^*)^*$.

   $(2)$ This follows directly from Propositions \ref{starting0P} and \ref{minimal-appro-projective-dual}.
\end{proof}
We recall that  an indecomposable  module $M$ is preprojective if there is an indecomposable projective module $P$ such that $M=\tau^{-t} P$ for some $t\geq 0$. An arbitrary module is called preprojective if it is a direct sum of indecomposable preprojective A-modules. Let $\CP r(\La)$ denote  the subcategory of $\mmod \La$ consisting of all preprojective modules. Dually,   preinjective modules can be defined, and we denote by $\CP i(\La)$ the subcategory of $\mmod \La$ consisting of all preinjective modules.

\begin{proposition}
The equivalence ${\rm Tr}$ is restricted to an equivalence form $\underline{\CP r(\La)}$ to $\underline{\CP i(\La^{\rm op})}$.
\end{proposition}
\begin{proof}
In view of subsection \ref{FromArtoCP},  we can consider $\Gamma_{\La}$ as a full subquiver of $\Gamma_{\CP}$. Under this embedding, we observe that  indecomposable preprojective, resp. preinjective, modules  correspond to the $\tau_{\CP}$-orbits of indecomposable projective, resp. injective,  objects of the form $(0 \rt P)$, resp. $(P\rt 0)$, for some indecomposable projective module $P$. Since $(0 \rt P)^*=(P^*\rt 0)$, we deduce that under the duality $(-)^*:\CP(\La)\rt \CP(\La^{\rm op})$ the $\tau_{\CP}$-orbit of $(0 \rt P)$ is mapped to the $\tau_{\CP}$-orbit of $(P^*\rt 0)$. Next, using the  following  description of the transpose functor ${\rm Tr}$ in terms of the morphism categories, as given in  \ref{Tran-Eq},
 $$ \underline{{\rm mod}}\mbox{-} \La \st{\overline{\rm Cok}^{-1}_{\La}} \rt \CP(\La)/\mathcal{Y}\st{(-)^*}\rt \CP(\La^{\rm op})/\mathcal{Y}'\st{\overline{\rm Cok}_{\La^{\rm op}}}\rt \underline{{\rm mod}}\mbox{-} \La^{\rm op},$$
\noindent
we can prove  the result.
\end{proof}

\subsection{minimal projective presentation}
In Proposition  \ref{Cok-preserve-almost-split},  we investigated sending  an almost split sequence in $\CP$ to $\mmod \La$ by taking the cokernel. Below, we will explore the inverse process by taking the minimal projective presentation.

To do so, we first require  the following construction from \cite[Appendix A]{E}.

\begin{construction}\label{projective Cover}Let ${\rm X}=(M\st{f}\rt N)$ be an object in $\mathcal{H}$,  and let $P\st{\alpha}\rt M$ and   $Q\st{\beta} \rt \text{Cok}(f)$ be projective covers. By the  projectivity, there exists $\delta:Q\rt N$ such that $p \delta=\beta$, where $p:N\rt \text{Cok}(f)$ is the canonical epimorphism. Then, by \cite[Theorem A1]{E}, the morphism
	$$\left(\begin{smallmatrix}
\alpha\\ \left[\begin{smallmatrix}
f\alpha & \delta
\end{smallmatrix}\right]
\end{smallmatrix}\right):\left(\begin{smallmatrix}
P\\ P\oplus Q
\end{smallmatrix}\right)_{\left[\begin{smallmatrix}
	1\\ 0
	\end{smallmatrix}\right]}\rt \left(\begin{smallmatrix}M\\  N\end{smallmatrix}\right)_{f}$$	
provides the projective cover of ${\rm X}$ in $\mathcal{H}$.

\end{construction}

We will apply the above construction to construct  the minimal projective presentation in $\CH$.
\begin{construction}\label{Construction1}
    Let $\epsilon: 0 \rt  L \st{a}\rt N \st{b}\rt M \rt 0 $ be a short exact sequence in $\mmod \La$. By using the horseshoe lemma, we obtain  the following commutative diagram with exact columns and rows:
\begin{equation}\label{Diagram11}
\xymatrix{	0 \ar[r] & P' \ar[d]^g \ar[r] & P'\oplus P \ar[d]^h
\ar[r] &  P\ar[d]^f \ar[r]\ & 0\\
	0 \ar[r] & Q' \ar[d]^l \ar[r] & Q'\oplus Q
	\ar[d]^d \ar[r] & Q \ar[d]^p \ar[r]\ar@{.>}[ld]^{\delta} & 0\\
0 \ar[r]  & L \ar[d] \ar[r]^a & N
	\ar[d] \ar[r]^b & M  \ar[d] \ar[r] & 0\\
& 0  & 0  & 0 & }
\end{equation}
In view of Remark \ref{Remark-Split}, we can present  the morphism $h$ as  ${h}= \tiny {\left[\begin{array}{ll} f & q \\ 0 & g \end{array} \right]}$ with $q:P \rt Q'$. The first two top rows give rise to  the following short exact sequence in $\CP$
$$\xymatrix@1{ E: \ \ 0\ar[r] & {\left(\begin{smallmatrix} P'\\ Q'\end{smallmatrix}\right)}_{g}
			\ar[r]
			&  {\left(\begin{smallmatrix}P'\oplus P\\ Q'\oplus Q\end{smallmatrix}\right)}_{h}\ar[r]&
			{\left(\begin{smallmatrix}P\\Q\end{smallmatrix}\right)}_{f}\ar[r]& 0 }.$$

We  can interpret the above construction in terms of minimal protective presentations in the morphism category as follows: we can write $d=[al ~~\delta]$, where $\delta:Q\rt N$ with  $p=b \delta.$  According to Construction \ref{projective Cover},  we have the following projective cover in $\CH$

\begin{equation}\label{firstsyzygy11}
\left(\begin{smallmatrix}
l\\ d
\end{smallmatrix}\right):\left(\begin{smallmatrix}
Q'\\ Q'\oplus Q
\end{smallmatrix}\right)_{\left[\begin{smallmatrix}
	1\\ 0
	\end{smallmatrix}\right]}\rt \left(\begin{smallmatrix}L\\  N\end{smallmatrix}\right)_{a}
\end{equation}
We can write $h$ as the composite of $j$ and $d'$ as in the following diagram:
\begin{equation*}\label{Diagram1}
\xymatrix{	0 \ar[r] & P' \ar[d]^{l'}\ar[r] & P'\oplus P \ar[d]^{d'}
\ar[r] &  P\ar[d]^{p'} \ar[r]\ar@{.>}[ld]^{\delta'} & 0\\
	0 \ar[r] & \Omega_{\La}(L) \ar[d]^i \ar[r]^{\Omega a} & {\rm Ker}(d)
	\ar[d]^j \ar[r] & \Omega_{\La} (M) \ar[d]^k \ar[r] & 0\\
0 \ar[r]  & Q'  \ar[r] & Q'\oplus Q
	 \ar[r] & Q   \ar[r] & 0 }
\end{equation*}
where $g=i l', h=jd', f=kp'$ and $d'=[l'~~\delta']$. Moreover, again applying Construction \ref{projective Cover}, we have  the following projective cover in $\CH$

 \begin{equation}\label{firstsyzygy2}
 \left(\begin{smallmatrix}
l'\\ d'
\end{smallmatrix}\right):\left(\begin{smallmatrix}
P'\\ P'\oplus P
\end{smallmatrix}\right)_{\left[\begin{smallmatrix}
	1\\ 0
	\end{smallmatrix}\right]}\rt \left(\begin{smallmatrix}  \Omega_{\La}(L) \\ {\rm Ker} (d)\end{smallmatrix}\right)_{\Omega a}
\end{equation}

Since $\Omega_{\CH}(L\st{a}\rt N)=(\Omega_{\La}(L)\st{\Omega a}\rt {\rm Ker}(d))$, by combining  \ref{firstsyzygy11} and \ref{firstsyzygy2}, we obtain the following minimal projective presentation of $\left(\begin{smallmatrix}
L\\ N
\end{smallmatrix}\right)_a$ in $\CH$

\begin{equation}\label{firstsyzygy1}
\left(\begin{smallmatrix}
P'\\ P'\oplus P
\end{smallmatrix}\right)_{\left[\begin{smallmatrix}
	1\\ 0
	\end{smallmatrix}\right]}\st{\left(\begin{smallmatrix}
g\\ h
\end{smallmatrix}\right)}\rt\left(\begin{smallmatrix}
Q'\\ Q'\oplus Q
\end{smallmatrix}\right)_{\left[\begin{smallmatrix}
	1\\ 0
	\end{smallmatrix}\right]}\st{\left(\begin{smallmatrix}
l\\ d
\end{smallmatrix}\right)}\rt \left(\begin{smallmatrix}L\\  N\end{smallmatrix}\right)_{a} \rt 0.
\end{equation}
\end{construction}

 Viewing the upper half part of Diagram \ref{Diagram11} as the minimal projective presentation of an object in $\CH$ helps us to prove the following uniqueness, up to isomorphism.

\begin{lemma}\label{Lemma-minim-Mor}
 Suppose  $r:P'\oplus P\rt Q'\oplus Q$ is  another morphism that  satisfies the commutativity in Diagram \ref{Diagram11}. Then, there exists  isomorphisms $\phi_i, i=1,2,3, 4$  such that  the following  diagram commutes:

\[\xymatrix@C 1.1pc@R 1.5pc{ & P \ar[rr] \ar@{<-}[dl]^{\phi_3} \ar@{->}[dd]^>>>>{f} && P\oplus P' \ar@{<-}[dl]^{\phi_4} \ar@{->}[dd]^>>>>>{r}\\
P\ar@{->}[dd]^>>>>>>>>f \ar@{->}'[r][rr] && P\oplus P'\ar@{->}[dd]^>>>>>>>>{h}  &\\
& Q \ar[rr] \ar@{<-}[dl]_>>>>>{\phi_1}  && {Q\oplus Q'}\ar@{<-}[dl]^{\phi_2}\\
Q \ar@{->}[rr] && Q\oplus Q'&
}\]
\end{lemma}

\begin{proof}
Both morphisms, $h$ and $r$, induce a minimal projective presentation of $(L\xrightarrow{a}N)$, as detailed in the preceding construction. The uniqueness of these minimal projective presentations within $\CH$ leads directly to the stated result.
\end{proof}
Therefore, by  the above lemma, we can conclude that $h$ in the above diagram  can be determined up to isomorphism.

\hspace{ 1 mm}

Furthermore, suppose that  the short exact sequence $\epsilon$ in Construction \ref{Construction1} is an almost split sequence.  According to Lemma \ref{Lemma-minim-Mor} and Theorem \ref{Cok-ass}, we can know that the short exact sequence $E$ is almost split in $\CP.$  Since $\nu:\CP(\La)\rt \CI(\La)$ is an equivalence, as shown in  \ref{injective-morphism}, we get  the following almost split sequence in $\CI(\La)$
$$\xymatrix@1{ \nu E:   \ \ 0\ar[r] & {\left(\begin{smallmatrix} \nu P'\\ \nu Q'\end{smallmatrix}\right)}_{\nu g}
			\ar[r]
			&  {\left(\begin{smallmatrix}\nu P'\oplus \nu P\\ \nu Q'\oplus \nu Q\end{smallmatrix}\right)}_{\nu h}\ar[r]&
			{\left(\begin{smallmatrix}\nu P\\\nu Q\end{smallmatrix}\right)}_{\nu f}\ar[r]& 0 }.$$
Moreover, since $(-)^*:\CP(\La)\rt \CP(\La^{\rm op})$ is a duality, we obtain  the following almost split sequence in $\CP(\La^{\rm op})$

$$\xymatrix@1{ E^*: \ \ 0\ar[r] & {\left(\begin{smallmatrix} Q^*\\ P^*\end{smallmatrix}\right)}_{f^*}
			\ar[r]
			&  {\left(\begin{smallmatrix}(Q')^*\oplus Q^*\\ (P')^*\oplus P^*\end{smallmatrix}\right)}_{h^*}\ar[r]&
			{\left(\begin{smallmatrix}(Q')^*\\(P')^*\end{smallmatrix}\right)}_{g^*}\ar[r]& 0 }.$$

\begin{proposition}
 Keeping in mind the above notation.   Let $\epsilon: 0 \rt  L \st{a}\rt N \st{b}\rt M \rt 0 $ be an almost split sequence in $\mmod \La$.  Assume that $L$ is not projective.  Then the following assertions hold.
    \begin{itemize}
        \item [$(i)$] $0 \rt \tau L \st{\tau a}\rt  \tau N \st{\tau b}\rt \tau M\rt 0$
 is an almost split sequence in $\mmod \La$, which is obtained by applying  the kerenl functor on $\nu E.$
        \item [$(ii)$] $0 \rt {\rm Tr} L \st{{\rm Tr} a}\rt  {\rm Tr} N \st{\rm Tr b}\rt {\rm Tr} M\rt 0$
 is an almost split sequence in $\mmod \La$, which is obtained by applying  the cokernel  functor on $E^*.$
    \end{itemize}
\end{proposition}
\begin{proof}
    $(i)$ and $(ii)$ are  immediate consequences of Propositions \ref{Cok-preserve-almost-split} and \ref{Cok-ass}, respectively. Note that the assumption of  $L$ not being projective guarantees  the objects $(\nu P'\st{\nu g}\rt \nu Q')$ and $((Q')^*\st{f^*}\rt (P')^*)$ satisfy the required conditions of  the mentioned  propositions.
\end{proof}
We have seen in Construction \ref{Construction1} that  we can construct   a projective presentation  for $N$ by using the minimal projective presentation of the ending terms. However, in general,   the obtained projective presentation is not minimal. In the next result, we will consider when the given short exact sequence is an almost split sequence.
\begin{proposition}\label{almost-split}
 Let $\epsilon: 0 \rt  L \st{a}\rt N \st{b}\rt M \rt 0 $ be an almost split sequence in $\mmod \La$. Then, $P'\oplus P\st{h}\rt Q'\oplus Q\rt N\rt 0 $ is a minimal projective presentation  if and only if
\begin{itemize}
\item [$(i)$]  $L  $ is not  a simple module
\item [$(ii)$]  $L$ is not  isomorphic to a direct summand of  $\nu X/{\rm soc}(\nu X)$, for some projective module $X$.
\end{itemize}
\end{proposition}
\begin{proof}
According to Proposition \ref{Cok-preserve-almost-split}, we have the following almost split sequence in $\CP$
$$\xymatrix@1{ E: \ \ 0\ar[r] & {\left(\begin{smallmatrix} P'\\ Q'\end{smallmatrix}\right)}_{g}
			\ar[r]
			&  {\left(\begin{smallmatrix}P'\oplus P\\ Q'\oplus Q\end{smallmatrix}\right)}_{h}\ar[r]&
			{\left(\begin{smallmatrix}P\\Q\end{smallmatrix}\right)}_{f}\ar[r]& 0 }.$$
We know that $h$ provides a minimal projective presentation for $N$ if and only if the object  $(P'\oplus P\st{h}\rt Q'\oplus Q)$ has neither direct summand isomorphic to $(X\st{1}\rt X)$ nor $(X\rt 0)$, for some indecomposable projective module $X$. Thanks to Lemma \ref{righ-left-minimimal-proj-obj}, the object $(P'\st{g}\rt Q')$  induces the minimal projective presentation of a simple module if and only if the middle term of the almost split sequence $E$ has a direct summand isomorphic to $(X\st{1}\rt X).$ On the other hand, the  object $h$ has a direct summand isomorphic to $(X\rt 0)$, for some projective module $X$, if and only if the object $(P'\st{g}\rt Q')$  is isomorphic to a direct summand of the middle term of the almost split sequence in $\CP$ ending at $(X\rt 0)$. Let ${\rm W} $ denote the middle term of the almost split sequence ending at $(X \rt 0)$. From Corollary \ref{CorollaryCokernel}, we know  that the cokernel of the middle term  ${\rm W}$ is $\nu X/ {\rm soc}(\nu X)$. Hence the result is proved.
\end{proof}

\subsection{$g$-vectors}
Assume that  $\CC$ is a Krull-Schmidt  category. Let ${\rm K}_0(\CC, 0)$ be the free abelian group $\oplus_{[X] \in {\rm ind }\mbox{-}\CC}\mathbb{Z}[X]$ generated by the set ${\rm ind}\mbox{-}\CC$ of isomorphism classes of indecomposable objects in $\CC$. Now, let us  assume that $\CC$ is an  exact category. We denote by ${\rm Ex}(\CC)$ the subgroup of ${\rm K}_0(\CC, 0)$ generated by
$$\langle [X]-[Y]+[Z] \,| \  X {\rightarrowtail} Y {\twoheadrightarrow} Z \ \text{\ is an short exact sequence in} \ \CC \rangle$$
We call the quotient group ${\rm K}_0(\CC):={\rm K}_0(\CC, 0)/{\rm Ex}(\CC)$ {\em the Grothendieck group} of $\CC$.

Let ${\rm K}_0({\rm prj}\mbox{-}\La)$ and ${\rm K}_0(\CP)$ denote the Grothendieck groups of the exact categories ${\rm prj}\mbox{-}\La$ and $\CP$, respectively. We have ${\rm K}_0({\rm prj}\mbox{-}\La, 0)={\rm K}_0({\rm prj}\mbox{-}\La)$, here ${\rm prj}\mbox{-}\La$ equipped with the split exact structure. There exists a well-defined  homomorphism group  ${\rm K}_0({\rm prj}\mbox{-}\La) \to {\rm K}_0(\CP)$ given by
\begin{align*}
& [(P\st{f} \rt Q)]\mapsto [Q]-[P]
\end{align*}
\noindent
Note that in view of Remark \ref{Remark-Split}, we can see $\Psi$ is well-defined.

\hspace{ 1 mm}

 Recall from \cite{DK} that the $g$-vectors are defined as follows:  Let $|\La|=n$ and let  $P_1,\cdots, P_n$ be  a complete set of pairwise non-isomorphic indecomposable projective modules.  It is well-known  that $[P_1], \cdots, [P_n]$   form
a basis of ${\rm K}_0({\rm prj}\mbox{-}\La)$. Consider $M$ in $\mmod \La$ and let
$$P_1\rt   P_0 \rt M \rt 0$$
be its minimal projective presentation in $\mmod \La$. Then we write
$$[P_0]-[P_1]=\Sigma^n_{i=1}g_i^M[P_i]$$
where, by definition,  $g^M=(g_i^M, \cdots, g_n^M)$ is the $g$-vector of $M$. It is clear  for every object  $(P\st{f}\rt Q)$ without summand of the form $(X\rt 0)$, we have $\Psi([P\st{f}\rt Q])=g^{{\rm Cok}(f)}$.

\begin{proposition}
     Let $\epsilon: 0 \rt  L \st{a}\rt N \st{b}\rt M \rt 0 $ be an almost split  sequence  in $\mmod \La$. If $L$ is not isomorphic to a direct summand of  $\nu X/{\rm soc} (\nu X)$, for some projective module $X$,  then
     $$g^N=g^L+g^M.$$
\end{proposition}
\begin{proof}
    Assume that $L$ is not isomorphic to a direct summand of  $\nu X/{\rm soc} \ \nu X$, for some projective module $X$. Consider minimal projective presentations  $P \st{f} \rt Q\rt M\rt 0$ and $P'\st{g}\rt Q'\rt L\rt 0$. In view of Construction \ref{Construction1}, there exists the following short exact sequence in $\CP$
$$\xymatrix@1{ 0\ar[r] & {\left(\begin{smallmatrix} P'\\ Q'\end{smallmatrix}\right)}_{g}
			\ar[r]
			&  {\left(\begin{smallmatrix}P'\oplus P\\ Q'\oplus Q\end{smallmatrix}\right)}_{h}\ar[r]&
			{\left(\begin{smallmatrix}P\\Q\end{smallmatrix}\right)}_{f}\ar[r]& 0 }.$$
\noindent
If $L$ is not a simple module, then the middle term gives a minimal projective presentation for $N$, by Proposition \ref{almost-split}. Otherwise,  we have  a decomposition of the middle term $h$ as
  $$(U\st{l}\rt V) \oplus (Y\st{1}\rt Y ),$$
  where $U\st{l}\rt V\rt N\rt 0$ a minimal projective presentation, and $Y$ is a projective module.   By considering the above  short exact sequence , we have in ${\rm K}_0({\rm prj}\mbox{-}\La)$
  \begin{align*}
      g^N&=[V]-[U]\\
      & =[V\oplus Y]- [U\oplus Y]\\
      &=[Q'\oplus Q]-[P'\oplus P]\\
      &=([Q']-[P'])+([Q]-[P])\\
      &=g^L+g^M.
  \end{align*}
\end{proof}

Therefore, the $g$-vectors behave additively with respect to almost split sequences except for a finite number of exceptions.

\begin{proposition}
    Let $\La$ be a representation-finite algebra with an acyclic AR-quiver. Assume $I_1, \cdots I_n$ is a complete set of indecomposable injective module up to isomorphism. Then, the $g$-vectors $g_1^{I_1}, \cdots, g^{I_n}$ generates ${\rm K}_0({\rm prj}\mbox{-}\La)$.
\end{proposition}
\begin{proof}
First,  note that by  \ref{FromArtoCP} and our assumption,    the AR-quiver  $\Gamma_{\CP}$ is acyclic as well. Let $P_1,\cdots, P_n$ be  a complete set of pairwise non-isomorphic indecomposable projective modules. We need to prove that all $[P_i]$'s can be expressed as a $\mathbb{Z}$-linear combination of $[g^{I_j}]$'s. Take an integer  $1 \leq i \leq n$.  From Proposition \ref{minimal-appro-projective-dom}, we have the following almost split exact sequence
\begin{equation}\label{the last-Eq}
\xymatrix@1{  0\ar[r] & {\left(\begin{smallmatrix} P\\ Q\end{smallmatrix}\right)}_{g}
			\ar[rr]^-{\left(\begin{smallmatrix} [1~~0]^t\\ 1\end{smallmatrix}\right)}
			& & {\left(\begin{smallmatrix}P\oplus P_i\\  Q\end{smallmatrix}\right)}_{[g~~e]}\ar[rr]^-{\left(\begin{smallmatrix} [0~~1] \\ 0\end{smallmatrix}\right)}& &
		{\left(\begin{smallmatrix}P_i\\0\end{smallmatrix}\right)}_{0}\ar[r]& 0 },
  \end{equation}
where $Q\st{g}\rt P\rt \nu P_i\rt 0$ is  a minimal projective presentation. The short exact  sequence gives the following equality in ${\rm K}_0(\CP)$

$$[P\oplus P_i \st{[g~~e]}\rt Q]=[P\st{g}\rt Q]+[P_i\rt 0].$$ Applying the group homorphism $\Psi$ on the above equality, we get
\begin{align}\label{equ}
    [P_i]&=g^{\nu P_i}-\Psi([P\oplus P_i \st{[g~~e]}\rt Q]).
\end{align}
We decompose the object $(P\oplus P_i \st{[g~~e]}\rt Q)$ as
$$(P\oplus P_i \st{[g~~e]}\rt Q)=(A\st{f}\rt B)\oplus (X\st{d}\rt Y)\oplus {\rm W}$$
where ${\rm Cok}(f)\simeq \nu K$ for some projective module $K$, $d$ is an isomorphism, and ${\rm W}$ has no direct summand isomorphic to either $(Y\st{1}\rt Y)$ or $(Z\st{r}\rt V)$ with injective module ${\rm Cok}(r)$. Then, we can extend  \ref{equ} as below
\begin{align*}
    [P_i]&=g^{\nu P_i}-\Psi([P\oplus P_i \st{[g~~e]}\rt Q])\\
    &= g^{\nu P_i}-\Psi(A\st{f}\rt B ) -\Psi(X\st{d}\rt Y ) -\Psi({\rm W})\\
   &=g^{\nu P_i}-g^{\nu K}-\Psi({\rm W})   \ \
\end{align*}
  Note that since  $X\simeq Y$, we get $\Psi(X\st{d}\rt Y)=0$. If ${\rm W}=0$, then we are done. Otherwise, let  ${\rm W}={\rm W}_1\oplus \cdots \oplus {\rm W}_m$ be a decomposition of ${\rm W}$ into indecomposable modules. Hence $\Psi({\rm W})=\Psi({\rm W}_1)+\cdots \Psi({\rm W}_m)$. For any direct summand ${\rm W_i}$ that  is non-injective in $\CP$,  there is an almost split sequence  $0 \rt {\rm W}_i\rt B_i\rt \tau^{-1}_{\CP} {\rm W}_i\rt 0$ in $\CP$. Due to this almost split sequence, we  have $\Psi({\rm W}_i)=\Psi(B_i)-\Psi(\tau^{-1}_{\CP} {\rm W}_i)$. If ${\rm W}_i$ is injective, two cases hold:
  \begin{itemize}
  \item [$(1)$] ${\rm W}_i=(U\st{1}\rt U)$, in which case  there is nothing to prove, as $\Psi({\rm W}_i)=0$.
  \item [$(2)$] ${\rm W}_i=(U\rt 0)$, in which case there is an almost split sequence  as  given in \ref{the last-Eq} with the middle term ${\rm X}_i$ and the first term with the cokernel $\nu E$, for some injective module. Hence we can replace $\Psi({\rm W}_i)$ with $\Psi({\rm X}_i)+g^{\nu E}$ in ${\rm K}_0({\rm prj}\mbox{-}\La)$.
  \end{itemize}

  We apply  the same process to the new terms $\Psi(B_i), \Psi(\tau^{-1}_{\CP} {\rm W}_i)$ and $\Psi({\rm X}_i)$. As $\Gamma_{\CP}$ is acyclic, the process will stop and   we will obtain  the desired $\mathbb{Z}$-linear combination.
\end{proof}

\end{document}